\documentclass{amsart}

\usepackage[utf8]{inputenc}
\usepackage{amssymb}
\usepackage[leqno]{amsmath}
\usepackage{latexsym}
\usepackage{amsthm}
\usepackage{enumitem}
\usepackage{graphicx}
\usepackage{epstopdf}
\usepackage{bigints}
\usepackage[all,cmtip]{xy}
\usepackage{esint}
\usepackage{stmaryrd}
\setlist[itemize]{leftmargin=4ex}

\theoremstyle{plain}
\newtheorem{theorem}{Theorem}
\newtheorem{corollary}[theorem]{Corollary}
\newtheorem{lemma}[theorem]{Lemma}
\newtheorem{proposition}[theorem]{Proposition}

\theoremstyle{remark}
\newtheorem{remark}[theorem]{Remark}

\numberwithin{equation}{section}
\numberwithin{theorem}{section}

\DeclareMathOperator{\diam}{diam}                              
                              
\newcommand{\N}{\mathbb{N}}
\newcommand{\R}{\mathbb{R}}

\newcommand{\Cleq}{\ensuremath{\lesssim}}  
\newcommand{\Cgeq}{\ensuremath{\gtrsim}}  
\DeclareMathOperator*{\infimum}{inf\vphantom{p}}               
               
\DeclareMathOperator{\cond}{cond}

\newcommand{\Mesh}{\mathcal{M}}
\newcommand{\Faces}[1]{\mathcal{F}_{#1}}
\newcommand{\LagNod}[1]{\mathcal{L}_{#1}}
\newcommand{\FacesM}{\Faces{}}

\newcommand{\FacesMint}{\FacesM^i}

\newcommand{\FacesMbnd}{\FacesM^b}

\newcommand{\LagNodM}[1]{\LagNod{#1}}

\newcommand{\LagNodMint}[1]{\LagNodM{#1}^i}
\newcommand{\LagNodMbnd}[1]{\LagNodM{#1}^b}

\newcommand{\Jump}[2]{\llbracket #1 \rrbracket_{#2}}
\newcommand{\Shape}{\gamma}

\DeclareMathOperator{\GradM}{\nabla_{\Mesh}}             
\DeclareMathOperator{\Hess}{D^2}
\DeclareMathOperator{\HessM}{D^2_{\Mesh}}                
\newcommand{\App}{P}   
\newcommand{\app}{M}										   

\newcommand{\Ritz}{\Pi}                                          
\newcommand{\id}{\mathrm{Id}}                                  
\newcommand{\AvgOper}[1]{A_{#1}}
\newcommand{\BubbOper}[1][1]{B_{#1}}

\newcommand{\snorm}[1]{\left \lvert #1 \right \rvert}          
\newcommand{\norm}[1]{\| #1 \|}                     
\newcommand{\opnorm}[3]{\norm{#1}_{\ifx#2#3\mathcal{L}(#2)\else\mathcal{L}(#2,#3)\fi}}       
            
\newcommand{\Domain}{\Omega}
\newcommand{\Leb}[1]{L^2(#1)}                                
                              
\newcommand{\SobH}[1]{H^1_0(#1)}

\newcommand{\SobbH}{H^2_0(\Domain)}

\newcommand{\degree}{p}
\newcommand{\Poly}[1]{\mathbb{P}_{#1}}                              
\newcommand{\Polyell}[1]{S_{#1}^0}
\newcommand{\Lagr}[1]{S_{#1}^1}
\newcommand{\crs}{\mathit{CR}}
\newcommand{\CRS}[1]{\crs_{#1}}
\newcommand{\mrs}{\mathit{MR}}
\newcommand{\MRS}[1]{\mrs_{#1}} 
\newcommand{\gls}{\mathit{CR}}
\newcommand{\GL}[1]{\gls_{#1}}

\newcommand{\hct}{\mathit{HCT}}
\newcommand{\HCT}{\hct}
\newcommand{\rhct}{\mathit{rHCT}}

\newcommand{\Cqopt}{C_{\mathrm{qopt}}}
\newcommand{\Cstab}{C_{\mathrm{stab}}} 
\newcommand{\aext}{\widetilde{a}}

\newcommand{\Vext}{\widetilde{V}}                         
\newcommand{\vext}{\widetilde{v}}

\begin{document}

\title[Quasi-optimal nonconforming methods II]
{Quasi-optimal nonconforming methods for symmetric elliptic problems.
II -- Overconsistency and classical nonconforming elements}

\author[A.~Veeser]{Andreas Veeser}
\author[P.~Zanotti]{Pietro Zanotti}


\begin{abstract}
We devise variants of classical nonconforming methods for symmetric elliptic problems. These variants differ from the  original ones only by transforming discrete test functions into conforming functions before applying the load functional. We derive and discuss  conditions on these transformations implying that the ensuing method is quasi-optimal and that its quasi-optimality constant coincides with its stability constant. As applications, we consider the approximation of the Poisson problem with Crouzeix-Raviart elements and higher order counterparts and the approximation of the biharmonic problem with Morley elements. In each case, we construct a computationally feasible transformation and obtain a quasi-optimal method with respect to the piecewise energy norm on a shape regular mesh. 
\end{abstract}

\maketitle

\section{Introduction}
\label{S:intro}
%
%
This article is the second in a series on the design and analysis of quasi-optimal nonconforming methods for symmetric elliptic problems. It concerns methods with classical nonconforming elements. The Crouzeix-Raviart element \cite{Crouzeix.Raviart:73} approximating the Poisson problem may be viewed as a prototypical example of such methods. Let us illustrate our motivation and main results in this case.

\medskip Let $\Mesh$ be a simplicial mesh of a domain $\Domain\subseteq\R^d$, $d\geq2$, and denote by $\FacesM$ the set of its $(d-1)$-dimensional faces. Furthermore, let $\CRS{}$ be the discrete space of real-valued functions on $\Omega$ that are piecewise affine, continuous in the midpoints of the internal faces of $\Mesh$ and vanish in the midpoints of boundary faces.  Since such functions can be discontinuous or nonzero in other points of the faces, $\CRS{}$ is not a subspace of the Sobolev space $H^1_0(\Omega)$. However, the Crouzeix-Raviart interpolant $\Ritz_{\crs}:H^1_0(\Omega) \to \CRS{}$, given by
\begin{equation}
\label{CR-interpolation}
\forall F \in \FacesM
\quad
\int_F \Ritz_{\crs} u = \int_F u,
\end{equation}
reveals remarkable approximation properties: for any function $u \in H^1_0(\Omega)$, we have
\begin{equation}
\label{CR-localization}
\begin{aligned}
\inf_{s \in \CRS{}} \norm{\GradM (u-s)}_{\Leb{\Domain}}
&=
\norm{\GradM (u-\Ritz_{\crs} u)}_{\Leb{\Domain}}
\\
&=
\left(
\sum_{K\in\Mesh} \inf_{p \in \Poly{1}(K)}
\norm{\nabla(u-p)}_{\Leb{K}}^2
\right)^{\frac{1}{2}},
\end{aligned}
\end{equation}
where $\GradM$ stands for the broken gradient. We see that, although the global best error of the Crouzeix-Raviart space is coupled or constrained at the midpoints of the faces, it is locally computable and exploits optimally the approximation capabilities of its shape functions. The latter improves on the space of continuous piecewise affine functions, which exploits the shape functions only in a quasi-optimal manner, depending on the shape coefficient of $\Mesh$; cf.\ Veeser \cite{Veeser:16}.

The space $\CRS{}$ is used in the homonymous method for the Poisson problem,
\begin{equation}
\label{org-CR-for-Poisson}
U \in \CRS{}
\quad\text{such that}\quad
\forall \sigma \in \CRS{}
\quad
\int_{\Omega} \GradM U \cdot \GradM\sigma
=
\int_{\Omega} f \sigma,
\end{equation}
where we suppose $f \in \Leb{\Domain}$. This is a nonconforming Galerkin method in the sense of the first part \cite{Veeser.Zanotti:17p1} of this series, because the underlying bilinear and linear forms on the conforming part $\CRS{}\cap H^1_0(\Omega)$ of the discrete space arise by simple restriction of their infinite-dimensional counterparts.

The question arises how much of the aforementioned remarkable approximation properties of the Crouzeix-Raviart space $\CRS{}$ are exploited in the method \eqref{org-CR-for-Poisson}. The so-called second Strang lemma \cite{Berger.Scott.Strang:72} yields
\begin{equation}
\label{2d-Strang}
 \norm{\GradM(u-U)}
 \approx
 \inf_{s \in \CRS{}} \norm{\GradM (u-s)}_{\Leb{\Domain}}
 +
 \text{CE}(u),
\end{equation}
where $\text{CE}(u)$ measures the consistency error induced by nonconforming discrete test functions. The survey \cite{Brenner:15} by S.~Brenner illustrates two approaches for bounding $\text{CE(u)}$: the classical one and the medius analysis initiated by Th.~Gudi \cite{Gudi:10}. Both bounds involve regularity beyond $H^1_0(\Omega)$: for example, the norm $\norm{D^2u}_{\Leb{\Domain}}$ of the Hessian for the classical approach and an $L^2$-oscillation of $\Delta u$ for the medius analysis. Remark~4.9 of \cite{Veeser.Zanotti:17p1} reveals that $\text{CE}(u)$ cannot be bounded only in terms of the best error $\inf_{s \in \CRS{}} \norm{\GradM (u-s)}_{\Leb{\Domain}}$. The reason for this lies in the fact that \eqref{org-CR-for-Poisson} applies nonconforming functions to the load $f$. Thus, the classical Crouzeix-Raviart method \eqref{org-CR-for-Poisson} is not quasi-optimal with respect to $\norm{\GradM\cdot}_{\Leb{\Domain}}$ and so does not always fully exploit the approximation properties of its underlying space $\CRS{}$. 

In order to remedy, we may consider, for a bounded linear smoothing operator $E:\CRS{}\to H^1_0(\Omega)$ to be specified, the following two variants of the original Crouzeix-Raviart method:
\begin{subequations}
\begin{align}
\label{mod-CR-for-Poisson}
&U_E \in \CRS{}
\quad\text{such that}\quad
\forall \sigma \in \CRS{}
\quad
\int_{\Omega} \GradM U_E \cdot \GradM\sigma
=
\langle f, E\sigma\rangle,
\\
\label{mod2-CR-for-Poisson}
&\bar{U}_E \in \CRS{}
\quad\text{such that}\quad
\forall \sigma \in \CRS{}
\quad
\int_{\Omega} \GradM \bar{U}_E \cdot \nabla E\sigma
=
\langle f, E\sigma\rangle.
\end{align}
\end{subequations}
Both variants are well-defined for arbitrary $f \in H^{-1}(\Omega) = \SobH{\Domain}'$ and each one has attractive features: the bilinear form of \eqref{mod-CR-for-Poisson} is symmetric, while the error of \eqref{mod2-CR-for-Poisson} is orthogonal to the range of $E$. Analyzing an abstract version of \eqref{mod2-CR-for-Poisson} with the tools from \cite{Veeser.Zanotti:17p1},  we find that its quasi-optimality constant depends only on the range of $E$ and that, for a fixed range, the energy norm condition number of its bilinear form becomes minimal, if $E$ is a right inverse of the best approximation operator onto $\CRS{}$. Notably, the two variants also coincide under this condition.

Combining  \eqref{CR-interpolation} and \eqref{CR-localization}, we see that $E$ is a right-inverse of the best approximation operator onto $\CRS{}$ if and only if 
\begin{equation*}
\forall \sigma \in \CRS{}, F \in \FacesM
\quad
\int_F E\sigma = \int_F \sigma.
\end{equation*}
Exploiting this local characterization, we construct a computationally feasible operator $E$ such that \eqref{mod2-CR-for-Poisson}, or equivalently \eqref{mod-CR-for-Poisson}, is quasi-optimal. More precisely, we have
\begin{equation*}
\norm{\GradM(u-U_E)}
\leq
\opnorm{E}{S}{V}
\inf_{s \in \CRS{}} \norm{\GradM (u-s)}_{\Leb{\Domain}},
\end{equation*}
where $\opnorm{E}{S}{V}$ is the best constant and equals the stability constant of resulting method. The construction of $E$, which is inspired by the one in Badia et al.\ \cite{Badia.Codina.Gudi.Guzman:14}, also ensures that $\opnorm{E}{S}{V}$ can be bounded in terms of the shape coefficient of the mesh $\Mesh$. It is also instrumental for designing quasi-optimal DG and other
interior penalty methods in the third part \cite{Veeser.Zanotti:17p3} of this series.

The rest of the article is organized as follows. In \S\ref{S:Groundwork} we recall relevant results of \cite{Veeser.Zanotti:17p1} and analyze well-posedness, conditioning and quasi-optimality of the abstract counterpart of \eqref{mod2-CR-for-Poisson}. In \S\ref{S:Applications} we then construct the aforementioned smoothing operator $E$, as well as similar operators when approximating the Poisson problem with Crouzeix-Raviart-like elements of arbitrary fixed order and 
the biharmonic problem with the Morley element.

In the discussion of the examples, we restrict ourselves to polyhedral Lipschitz domains and homogeneous essential boundary conditions. More general settings as well as numerical experiments will be presented elsewhere.
\section{Quasi-optimal and overconsistent nonconforming methods}
\label{S:Groundwork}
This section devises the approach to the design of quasi-optimal nonconforming methods, which is exemplified in the introduction \S\ref{S:intro}. A key feature of the ensuing methods is that their quasi-optimality constant is not affected by consistency.

\subsection{Quasi-optimality of nonconforming methods}
\label{S:Qopt-nconf-methods}
%
We first briefly summarize the results of \cite{Veeser.Zanotti:17p1}, focusing on one approach to measure nonconforming consistency.

We consider the following linear and symmetric elliptic problems. Given an infinite-dimensional Hilbert space $V$ with scalar product $a(\cdot, \cdot)$ and \emph{energy norm} $\norm{\cdot} = \sqrt{a(\cdot,\cdot)}$, let $V'$ be the topological dual space of $V$.  Denote by $\left\langle \cdot, \cdot\right\rangle$ the dual pairing of $V$ and $V'$ and by $\norm{\ell}_{V'} := \sup_{v\in V, \norm{v}\leq 1} \langle\ell,v\rangle$ the \emph{dual energy norm}
on $V'$. 
The \emph{continuous problem} is then: given $\ell \in V'$, find $u\in V$ such that
\begin{equation}
\label{ex-prob}
\forall v \in V
\quad
a(u, v) 
= 
\langle \ell, v \rangle.
\end{equation} 
%
This problem is well-posed in the sense of Hadamard and, introducing the Riesz isometry $A:V \to V'$, $v \mapsto a(v, \cdot)$, we have $u=A^{-1}\ell$ with
\begin{equation}
\label{isometry}
\norm{u} = \norm{\ell}_{V'}.
\end{equation}

We shall look for quasi-optimal methods in the following subclass of nonconforming linear variational methods for \eqref{ex-prob}. Let $S$ and $b$ the counterparts of $V$ and $a$, respectively. More precisely, let $S$ be a finite-dimensional linear space and $b : S \times S \to \R$ a nondegenerate bilinear form in that $ b(s,\sigma) = 0 $ for all $\sigma \in S$ entails $s=0$.
Remark~2.3 of \cite{Veeser.Zanotti:17p1} shows that a quasi-optimal method is necessarily \emph{entire}, i.e.\ defined for all $\ell \in V'$. Taking into account that we do not require $S \subseteq V$, we therefore introduce a linear operator $E:S \to V$ and define a linear operator $\app: V' \to S$ by the following \emph{discrete problem}: given $\ell \in V'$, find $\app\ell \in S$ such that
\begin{equation}
\label{disc-prob}
\forall \sigma \in S
\quad
b(\app \ell, \sigma)
=
\langle \ell, E \sigma \rangle,
\end{equation}
where we write $\left\langle \cdot, \cdot\right\rangle$ also for the pairing of $S$ and $S'$. We thus approximate the solution $u$ of \eqref{ex-prob} by $\app \ell$. Since $S\not\subseteq V$ often arises for the lack of smoothness, we refer to $E$ as a smoothing operator. Moreover, we identify $M$ with the triplet $(S,b,E)$, ignoring some slight ambiguity; see also \cite[Remark~2.2]{Veeser.Zanotti:17p1}.

The relationship between continuous and discrete problem is illustrated in Figure~\ref{F:EntireNonconformingMethods}.
\begin{figure}	
	\[
	\xymatrixcolsep{4pc}		
	\xymatrixrowsep{4pc}		
	\xymatrix{
		V'   									
		\ar[d]^{E^*}
		\ar[r]^{A^{-1}}       
		\ar[rd]_{M} 					
		& V 									
		\ar@{-->}[d]^{\App} \\		
		S'										
		\ar[r]^{B^{-1}}				
		& S}               		
	\]
	\caption{\label{F:EntireNonconformingMethods}	Diagram with operators $A$, $B$, $E$, nonconforming method $\app = (S, b, E)$, and induced approximation operator $\App$.}
\end{figure}
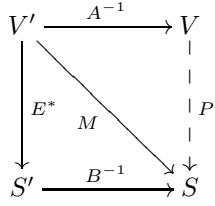
The commutative diagram involves also
\begin{itemize}
	\item the adjoint $E^* :V' \to S'$ given by $\left\langle E^* \ell, \sigma \right\rangle =  \left\langle \ell , E\sigma \right\rangle $ for all $\ell \in V'$, $\sigma \in S$,
	\item the invertible map $B:S \to S'$, $s \mapsto b(s,\cdot)$,
	\item  the approximation operator $\App := \app A$
\end{itemize} 
and reveals the representation
\begin{equation}
\label{M=}
\app = B^{-1} E^*.
\end{equation}

To assess the quality of the approximations given by $\app$, we assume that $a$ can be extended to a scalar product $\aext$ on $\Vext:=V+S$ and measure the error in the \emph{extended energy norm}
\begin{equation*}
 \norm{\cdot} := \sqrt{\aext(\cdot, \cdot)}
 \quad\text{on }\Vext,
\end{equation*}
with the same notation as for the original one.  The best approximation error within $S$ to some function $v \in V$ is then given by $\inf_{s\in S} \norm{v-s}$.  We say that the method $M$ is \emph{quasi-optimal} for \eqref{ex-prob} with respect to the extended energy norm if its approximations are are uniformly close to this benchmark, more precisely, if there exists a constant $C\geq 1$ such that
\begin{equation}
\label{Cqopt-def}
\forall u \in V
\qquad
\norm{ u - \App u }
\leq
C \inf_{s\in S} \norm{u-s}.
\end{equation}	
We denote by $\Cqopt$ the smallest constant in \eqref{Cqopt-def} and refer to it as the \textit{quasi-optimality constant} of $\app$.
To design quasi-optimal methods of the type $(S,b,E)$, our departure point is the following result.
\begin{theorem}[Stability, consistency, and quasi-optimality]
	\label{T:abs-qopt}
	Any nonconforming method $\app = (S,b,E)$ for \eqref{ex-prob} satisfies:
	\begin{itemize}
		\item[(i)] $\app$ is bounded, or fully stable, with
		\begin{equation*}
		\Cstab
		:=
		\opnorm{\app}{V'}{S}
		=
		\sup_{\sigma\in S} \, \frac{\norm{E\sigma}}{\sup_{s \in S, \norm{s} = 1} b(s,\sigma)}.
		\end{equation*}
		\item[(ii)] $\app$ is quasi-optimal if and only if it is fully algebraically consistent in that  
		\begin{equation*}
		\forall u \in S\cap V, \sigma \in S
		\quad
		b(u, \sigma)
		=
		a(u, E\sigma).  
		\end{equation*}
		\item[(iii)] If $\app$ is quasi-optimal, then its quasi-optimality constant is 
		\begin{equation*}
		\Cqopt
		=
		\sup_{\sigma \in S} \,
		\frac{ \sup_{v\in V, s \in S, \norm{v+s} = 1} a(v,E\sigma) + b(s,\sigma) }
		{\sup_{s \in S, \norm{s} = 1} b(s,\sigma)}
		\end{equation*}
		and satisfies
		\begin{equation*}
		\max \{\Cstab, \delta \}
		\leq
		\Cqopt
		\leq
		\sqrt{\Cstab^2+\delta^2},
		\end{equation*}
		where $\delta \in [0,\infty)$ is the consistency measure given by the smallest constant in
		\begin{equation*}
		\forall s, \sigma \in S
		\quad
		\snorm{b(s,\sigma) - \aext(s, E\sigma)}  
		\leq
		\delta
		\sup_{\hat{s} \in S, \norm{\hat{s}} = 1} b(\hat{s},\sigma)
		\inf \limits_{v \in V}
		\norm{s - v}.
		\end{equation*}
	\end{itemize} 
\end{theorem}

\begin{proof}
	Since $E$ is defined on the whole discrete space $S$ and bounded, the claims (i)-(iii) follow from Theorems~4.1, 4.2 and 4.3 of \cite{Veeser.Zanotti:17p1}.
\end{proof}

We restrict ourselves to a few comments on Theorem \ref{T:abs-qopt}; for a comprehensive discussion, see \cite{Veeser.Zanotti:17p1}. In items (i) and (ii), `fully' refers to the fact that all (and not only certain smooth) instances of the continuous problem \eqref{ex-prob} are involved, via $V'$ in item (i) and via $V$ in item  (ii). Representation \eqref{M=} and item (i) show that \emph{full stability} hinges on the property that $E$ maps all $S$ into $V$ and that the stability constant $\Cstab$ results from the interplay between $E$ and the discrete bilinear form $b$. Their relationship to the continuous bilinear form $a$ suitable for quasi-optimality is identified in item (ii) by the condition of \emph{full algebraic consistency}. This condition is equivalent to requiring that $M$ reproduces every solution of \eqref{ex-prob} which happens to be in $S$. 
Item (iii) generalizes the formula for the quasi-optimality constant in   \cite{Tantardini.Veeser:16} for conforming methods, showing that it may be also affected by consistency for truly nonconforming methods. Furthermore, we capture this effect by means of the quantity $\delta$, which is finite if and only if $M$ is algebraically fully consistent and almost insensitive to stability, see \cite[Remark 3.5]{Veeser.Zanotti:17p1}. We call a method $M$ \emph{(algebraically) overconsistent} whenever its consistency measure $\delta$ vanishes.

A simple manner to \emph{partially} satisfy the condition of algebraic consistency is restriction. More precisely, if
\begin{equation}
\label{NonConformingGalerkinMethod}
b_{|S_C \times S_C} = a_{|S_C \times S_C}
\quad\text{and}\quad
E_{|S_C} = \id_{S_C},
\end{equation}
where $S_C := S \cap V$ is the conforming subspace of $S$, then the condition of algebraic consistency holds for conforming test functions $\sigma \in S_C$. Methods with \eqref{NonConformingGalerkinMethod} generalize (conforming) Galerkin methods and we refer to them as \textit{nonconforming Galerkin methods}. They are natural candidates for quasi-optimal nonconforming methods, but, in contrast to conforming Galerkin methods, they are not completely determined by the continuous problem and the discrete space. 

\subsection{Overconsistency}
\label{S:smoothing-right-inv}
%
Assume that we are given $V$ and $a$ of the continuous problem \eqref{ex-prob} and a discrete space $S$, along with an extended scalar product $\aext$. Then, in view of Theorem \ref{T:abs-qopt}, the design of a quasi-optimal method on $S$ reduces to the task of finding a smoothing operator $E$ and a bilinear form $b$ implying full algebraic consistency. There are three possibilities to define $b$ in terms of $\aext$ and $E$:
\begin{equation*}
\aext(\cdot,\cdot),
\qquad
\aext(\cdot,E\cdot),
\qquad\text{and}\qquad
\aext(E\cdot,E\cdot).
\end{equation*}
Since the third option corresponds to a conforming Galerkin method on the range $T=R(E)$ of $E$ also when $S\not\subseteq V$, it is covered by standard theory. We therefore do not consider it here. The first two, truly nonconforming options separate the advantages of a conforming Galerkin method for \eqref{ex-prob}: the first one is a symmetric bilinear form, while the second one corresponds to overconsistency. Interestingly, the two options coincide and unify their advantages if and only if the smoothing operator $E$ is a right inverse for the $\aext$-orthogonal projection $\Ritz$ from $\Vext$ onto $S$ because of the identity $\aext(s,E\sigma) = \aext(s,\Ritz E \sigma)$
for all $s,\sigma \in S$.

Here we investigate the second option
\begin{equation}
\label{bE}
b_E(s,\sigma) := \aext(s,E\sigma),
\qquad
s, \sigma \in S,
\end{equation}
of overconsistency, which, partially, shall bring us back to the first option. Writing $\app_E$ as an abbreviation for $(S,b_E,E)$, the resulting discrete problem reads as follows: given any $\ell\in V'$, find $\app_E\ell\in S$ such that
\begin{equation}
\label{ME}
\forall \sigma \in S
\quad
\aext(\app_E\ell, E\sigma)
=
\langle \ell, E\sigma \rangle.
\end{equation} 
%
Since the test function $\sigma$ enters only via $E\sigma$, such a method can be viewed as a Petrov-Galerkin method over $S \times T$ with the conforming test space $T:=R(E)$. In other words, \eqref{ME} is equivalent to
\begin{equation*}
\forall \tau \in T
\quad
\aext(\app_E\ell, \tau)
=
\langle \ell, \tau \rangle.
\end{equation*}
Consequently, properties of the map $\app_E$ depend on $E$ only through its range $T=R(E)$. In what follows, we underline this aspect whenever applicable.
%
Let us start by examining the solvability and related properties of \eqref{ME}.

\begin{remark}[Injectivity of smoothing]
\label{R:injectivity-of-smoothing}
In view of \eqref{M=}, the injectivity of the smoothing operator $E$ is equivalent to the surjectivity of $\app$. In connection with a bilinear form $b_E$, it becomes a necessary condition for the well-posedness of \eqref{ME}.
\end{remark}

\begin{lemma}[Nondegeneracy of $b_E$]
\label{L:nondegeneracy-of-bE}
For any injective linear operator $E:S\to V$ with range $T=R(E)$, the following statements are equivalent:
\begin{subequations}
\label{bEnd}
\begin{align}
\label{bEGalerkin-nd}
		&b_E \text{ is nondegenerate on } S\times S,
		\\
		\label{bEPetrov-Galerkin-nd}
		&\aext(\cdot,\cdot) \text{ is nondegenerate on } S\times T,
		\\
		\label{bEapprox-nd}
		&\Ritz_{|T} \text{ is invertible},
		\\
		\label{bEgeometric-nd}
		&S \cap T^\perp = \{0\},
\end{align}
\end{subequations}
where $\Ritz$ stands for the $\aext$-orthogonal projection from $\Vext$ onto $S$. If $b_E$ is nondegenerate, then its energy norm condition number is given by
\begin{equation}
	\label{bEcond}
	\cond(b_E)
	=
	\opnorm{(\Ritz E)^{-1}}{T}{T} \opnorm{\Ritz E}{S}{S}
	\geq
	1,
\end{equation}
which is minimized by $E = (\Ritz_{|T})^{-1}$.
\end{lemma}

\begin{proof}
The claimed equivalences are essentially a special case of the inf-sup theory; we provide the details of their proofs for the sake of completeness. 
	
We first observe that $E$ is a linear isomorphism from $S$ to $T$, which implies $\dim S = \dim T$ as well as $\eqref{bEGalerkin-nd} \iff \eqref{bEPetrov-Galerkin-nd}$.
	
Next, we verify $\eqref{bEPetrov-Galerkin-nd} \implies \eqref{bEapprox-nd}$ and  let $\tau \in T$ with $\Ritz\tau = 0$. This yields $0 = \aext(s,\Ritz\tau) = \aext(s,\tau)$ for all $s \in S$ and so, using \eqref{bEPetrov-Galerkin-nd}, we see that $\tau = 0$.  Consequently, the kernel of $\Ritz_{|T}$ is trivial and the rank-nullity theorem yields that  $\Ritz_{|T}$ is a linear isomorphism from $T$ to $S$. 
	
To show $\eqref{bEapprox-nd}\implies \eqref{bEgeometric-nd}$, consider any $s \in S \cap T^\perp$. Then $\tau := (\Ritz_{|T})^{-1} s \in T$ thanks to \eqref{bEapprox-nd} and $0 = \aext(s,\tau) = \aext(s,(\Ritz_{|T})^{-1} s) = \aext( s, \Ritz(\Ritz_{|T})^{-1} s) = \aext(s,s)$ gives $s=0$.  Hence we have $S\cap T^\perp = \{0\}$.
	
We complete the proof of the equivalences by showing $\eqref{bEgeometric-nd} \implies \eqref{bEPetrov-Galerkin-nd}$. Since $\dim S = \dim T$, it suffices to check the nondegeneracy for the first argument of $\aext$, that is, given $s\in S$, $\aext(s,\tau) = 0$ for all $\tau \in T$ implies $s=0$. This condition is just a reformulation of \eqref{bEgeometric-nd}, so that the desired implication is verified.
	
Finally, assuming that $b_E$ is nondegenerate, we turn to \eqref{bEcond} and recall that the energy norm condition number of $b_E$ is given by
$	\cond(b_E)
	=
	C_E/\beta_E,
$	
where 
\begin{equation*}
	C_E
	:=
	\sup_{s,\sigma\in S} \frac{b_E(s,\sigma)}{\norm{s}\norm{\sigma}}
	\geq
	\infimum_{s\in S} \sup_{\sigma \in S}
	\frac{b_E(s,\sigma)}{\norm{s}\norm{\sigma}}
	=
	\infimum_{\sigma\in S} \sup_{s\in S}
	\frac{b_E(s,\sigma)}{\norm{s}\norm{\sigma}}
	=: 
	\beta_E > 0.
\end{equation*}
We claim that, for any $\sigma \in S$,
\begin{equation}
	\label{|.|_bE}
	\sup_{s\in S} \frac{b_E(s,\sigma)}{\norm{s}}
	=
	\norm{\Ritz E \sigma}.
\end{equation}
Indeed, if $s \in S$, the properties of $\Ritz$ and the Cauchy-Schwarz inequality yield $b_E(s,\sigma) = \aext(s,E\sigma) = \aext(s,\Ritz E\sigma) \leq \norm{s} \norm{\Ritz E \sigma}$, with equality for $s = \Ritz E \sigma$.  Exploiting \eqref{|.|_bE} in the definition of $C_E$ and the second expression for $\beta_E$, we conclude
\begin{equation*}
	\cond(b_E)
	=
	\frac{\sup_{\sigma \in S, \norm{\sigma} = 1} \norm{\Ritz E\sigma}}
	{\inf_{\sigma \in S, \norm{\sigma} = 1} \norm{\Ritz E\sigma}}
	=
	\opnorm{(\Ritz E)^{-1}}{S}{S} \opnorm{\Ritz E}{S}{S}.
	\qedhere 
\end{equation*}
\end{proof}

Next, ignoring computational feasibility, we characterize the existence of at least one smoothing operator $E$ giving rise to a nondegenerate bilinear form $b_E$. This characterization reveals that the search for right inverses is not restrictive and will be used in \cite{Veeser.Zanotti:17p3} to observe that all $b_E$ are degenerate for various nonconforming elements.
\begin{lemma}[Existence of nondegenerate $b_E$]
\label{L:ex-nd-bE}
For any discrete space $S$ and extended scalar product $\aext$, the following statements are equivalent:
\begin{subequations}
\label{ex-nd-bE}
\begin{align}
\label{ex-nd-bE;E}
		&\text{there is an injective $E:S\to V$ such that $b_E$ is nondegenerate},
\\
\label{ex-nd-bE;S}
		&S \cap V^\perp = \{0\},
\\
\label{ex-nd-bE;right inverse}
		&\text{$\Ritz_{|V}$ admits a right inverse}.
\end{align}
\end{subequations}
\end{lemma}

\begin{proof}
First, we verify \eqref{ex-nd-bE;E}$\implies$\eqref{ex-nd-bE;S}. Assume $E:S\to V$ is injective and such that $b_E$ is nondegenerate.  Using Lemma \ref{L:nondegeneracy-of-bE}, we infer $S \cap T^\perp = \{0\}$ for $T = R(E)$. Since $T \subseteq V$, we have $V^\perp \subseteq T^\perp$ and $S \cap V^\perp \subseteq S \cap T^\perp = \{0\}$, whence $S \cap V^\perp = \{0\}$.
	
To show the implication \eqref{ex-nd-bE;S}$\implies$\eqref{ex-nd-bE;right inverse}, we assume that $S \cap V^\perp = \{0\}$ and observe $s \in S \cap V^\perp \iff s \in S \cap \Ritz(V)^\perp$ with the help of $\aext(v,s)=\aext(\Ritz v,s)$ for all $v\in V$ and $s \in S$.  We thus infer $\Ritz(V) = S$ and can apply \cite[Theorem 2.12]{Brezis:11} to obtain: $\Ritz_{|V}$ admits a right inverse if and only if $N(\Ritz_{|V})$ admits a complement in $V$. Since $\Ritz$ is $\aext$-orthogonal, we have $N(\Ritz_{|V}) = S^\perp \cap V$, which has the complement $S \cap V$ in $V$. Hence \eqref{ex-nd-bE;right inverse} holds.
	
The missing implication \eqref{ex-nd-bE;right inverse}$\implies$\eqref{ex-nd-bE;E} is straight-forward. Let $E:S\to V$ be a right inverse of $\Ritz_{|V}$ and observe that $E$ and $\Ritz_{|R(E)}$ have to be injective.Thus, Lemma \ref{L:nondegeneracy-of-bE} provides \eqref{ex-nd-bE;E}.
\end{proof}
%
Let us now turn to stability and quasi-optimality of overconsistent methods.
\begin{theorem}[Overconsistent quasi-optimality]
	\label{T:ME}
	Let $E:S\to V$ be any injective smoothing operator with range $T=R(E)$.  If $S \cap T^\perp = \{0\}$, then the method $\app_E = (S,b_E,E)$ is quasi-optimal with
	\begin{equation*}
	\Cqopt
	=
	\opnorm{(\Ritz_{|T})^{-1}}{S}{V}
	=
	\Cstab.
	\end{equation*}
\end{theorem}

\begin{proof}
	Since $S \cap T^\perp = \{0\}$, Lemma \ref{L:nondegeneracy-of-bE} ensures that $b_E$ is nondegenerate. Furthermore, $\app_E$ is fully stable and overconsistent by construction and so Theorem~\ref{T:abs-qopt} shows that $\app_E$ is quasi-optimal with $\Cqopt=\Cstab$. We conclude by deriving
	\begin{equation}
	\label{ME_Cstab}
	\Cstab
	=
	\sup_{\sigma \in S} \frac{\norm{E\sigma}}{\norm{\Ritz E \sigma}}
	=
	\sup_{\tau \in T} \frac{\norm{\tau}}{\norm{\Ritz\tau}}
	=
	\sup_{\sigma \in S} \frac{\norm{(\Ritz_{|T})^{-1}\sigma}}{\norm{\sigma}}
	=
	\opnorm{(\Ritz_{|T})^{-1}}{S}{V}.
	\end{equation}
	by inserting \eqref{|.|_bE} into Theorem~\ref{T:abs-qopt}~(i) and exploiting that $E:S\to T$ and $\Ritz_{|T}$ are bijective. 
\end{proof}

\begin{remark}[Overconsistency and increasing nonconformity]
For overconsistent methods, the constants $\Cqopt=\Cstab$ grow with increasing nonconformity. To see this, let $\sigma \in S\setminus V$ with $\norm{\sigma}=1$ be a nonconforming direction and let $\alpha \in [0,\pi/2)$ be its angle with the closed subspace $V$ given by $\cos\alpha = \sup_{v\in V, \norm{v}=1} |\aext(v,\sigma)|>0$. Since $T=R(E) \subseteq V$, the angle between $\sigma\in S$ and $(\Ritz_{|T})^{-1} \sigma$ is bigger than $\alpha$. Hence $\aext(\sigma, (\Ritz_{|T})^{-1}\sigma) = \norm{\sigma}^2 = 1$ yields $\Cqopt\geq\norm{(\Ritz_{|T})^{-1} \sigma} \geq (\cos\alpha)^{-1}$.
\end{remark}

\begin{remark}[Possible overestimation of classical upper bound for $\Cqopt$]
\label{R:weakness-of-ub-Cqopt}
The first identity in \eqref{ME_Cstab} and $ \opnorm{E}{S}{V} =  \sup_{\norm{\sigma}=1}\sup_{\norm{\vext}=1}  \aext(\vext, E\sigma) =: \widetilde{C}_E$ yield
\begin{equation*}
	\Cqopt
	\leq
	\opnorm{(\Ritz E)^{-1}}{S}{S} \opnorm{E}{S}{V}
	=
	\frac{\widetilde{C}_E}{\beta_E},
\end{equation*}
where the right-hand side admits the classical form of an upper bound for the quasi-optimality constant. Notably, this bound depends on $E$ not only through its range $T=R(E)$ and, closely related, may be pessimistic if $E$ has singular values of different size.
\end{remark}

Neglecting the computational feasibility, our analysis of overconsistent methods does not reveal any disadvantage of restricting the search of smoothing operators to right inverses for the $\aext$-orthogonal projection $\Ritz$. On the contrary, the bilinear form is given by simple restriction of $\aext$, thus symmetric, and minimizes its energy norm condition number within smoothing operators of the same range. We therefore aim at invoking the following special case of Theorem~\ref{T:ME}.

\begin{corollary}[Smoothing with right inverses]
\label{C:ME-right-inverse}
Let $E^\vartriangle:S\to V$ be a right inverse for the $\aext$-orthogonal projection $\Ritz$ from $\Vext$ onto $S$. 
Then $\app_{E^\vartriangle} = (S,\aext_{|S \times S},E^\vartriangle)$ and it is a nonconforming Galerkin method if and only if $E^\vartriangle{}_{|S\cap V} = \id_{S \cap V}$. Moreover,  $\app_{E^\vartriangle}$ is quasi-optimal with
\begin{equation*}
	\Cqopt
	=
	\Cstab
	=
	\opnorm{E^\vartriangle}{S}{V}.
\end{equation*}
\end{corollary}

\section{Applications with classical nonconforming finite elements}
\label{S:Applications}
%
%
In light of Corollary~\ref{C:ME-right-inverse}, the key step for quasi-optimality is to find a right inverse $E^\vartriangle$ for the projection $\Ritz$ that provides $V$-smoothing, is suitably bounded and \emph{computationally feasible}.  In the context of finite element methods, the latter is given if, for the finite element basis  $\varphi_1,\dots,\varphi_n$ at hand, the evaluations $\langle \ell, E^\vartriangle\varphi_i\rangle$, $i=1,\dots,n$, can be implemented with $O(n)$ operations. In this section, we construct such right inverses not only for the setting considered in the introduction \S\ref{S:intro}, but also for elements of arbitrary fixed order and for fourth order problems.

\subsection{From discontinuous to continuous piecewise polynomials}
\label{S:setting-FEM}
%
In what follows, the discrete functions will be piecewise polynomials over simplicial meshes. This section introduces related notation and facts.

Let $d \in \N$ and $n \in \{0, \dots, d\}$.  An $n$-\emph{simplex} $C \subseteq \R^d$ is the convex hull of $n+1$ points $z_1, \dots, z_{n+1} \in \R^d$ spanning an $n$-dimensional affine space. The uniquely determined points $z_1, \dots, z_{n+1}$ are the vertices of $C$ and form the set $\LagNod{1}(C)$. If $n\geq1$, denote by $\Faces{C}$ the $(n-1)$-dimensional faces of $C$, which are the $(n-1)$-simplices arising by picking $n$ distinct vertices from $\LagNod{1}(C)$.  Given a vertex $z \in \LagNod{1}(C)$, its barycentric coordinate $\lambda_z^C$ is the unique first order polynomial on $C$ such that $\lambda_z^C(y) = \delta_{zy}$ for all $y \in \LagNod{1}(C)$. Then $0\leq\lambda_z^C\leq 1$ in $C$ and, for any multi-index $\alpha = (\alpha_z)_{z\in\LagNod{1}(C)} \in \N_0^{n+1}$,
\begin{equation}
\label{intregation-bary-coord}
 \int_C {\textstyle \prod_{z\in\LagNod{1}(F)} }(\lambda_z^{C})^{\alpha_z}
 =
 \frac{n!\alpha!}{(n+|\alpha|)!} \snorm{C},
\end{equation}
where $\snorm{C}$ is the $n$-dimensional Hausdorff measure in $\R^d$. We write $h_C := \diam(C)$ for the diameter of $C$, $\rho_C$ for the diameter of its largest inscribed $n$-dimensional ball, and $\Shape_C$ for its shape coefficient $\Shape_C :=  h_C / \rho_C$.

If $\degree\in\N_0$, we write $\Poly{\degree}(C)$ for the linear space of \emph{polynomials} on $C$ with (total) degree $\leq\degree$. A polynomial $P\in\Poly{\degree}(C)$ is determined by its point values at the Lagrange nodes $\LagNod{\degree}(C)$ of order $\degree$, which, for $\degree\geq2$, are given by $\big\lbrace x \in C \mid \forall z \in \LagNod{1}(C) \;
p\lambda_z^C (x) \in \N_0 \big\rbrace$. These nodes are nested in that $\LagNod{\degree}(F) = \LagNod{\degree}(C) \cap F$ for any face $F\in\Faces{C}$. Thus, the restriction $P_{|F}$ is determined by the `restriction' $\LagNod{\degree}(C) \cap F$ of the Lagrange nodes.
 
Let $\Domain\subseteq\R^d$ be an open, bounded, polyhedral and connected set with boundary $\partial\Omega$, which is assumed to be Lipschitz if $d\geq2$. Furthermore, let $\Mesh$ be a simplicial, face-to-face \emph{mesh} of $\Omega$.  More precisely, $\Mesh$ is a finite collection of $d$-simplices in $\R^d$ such that
$\overline{\Domain} = \bigcup_{K \in \Mesh} K$ and the intersection of two arbitrary elements $K_1, K_2 \in \Mesh$ is either empty or an $n$-simplex with $n\in\{0 \dots, d\}$ and $\LagNod{1}(K_1 \cap K_2) = \LagNod{1}(K_1) \cap \LagNod{1}(K_2)$. We let $\FacesM :=  \bigcup_{K \in \Mesh} \Faces{K}$ denote
the $(d-1)$-dimensional faces of $\Mesh$ and distinguish between boundary faces
$\FacesMbnd := \{ F \in \FacesM \mid F \subseteq \partial\Omega \}$ and interior faces $\FacesMint := \FacesM \setminus \FacesMbnd$.  The shape coefficient of $\Mesh$ is 
\begin{equation*}
\label{shape-Mesh}
 \Shape_\Mesh:=
 \max \limits_{K \in \Mesh} \Shape_K.
\end{equation*}
If not specified differently, $C_*$ stands for a function which is not necessarily the same at each occurrence and depends on a subset $*$ of $\{d,\Shape_\Mesh,\degree\}$, increasing in $\Shape_\Mesh$ and $\degree$ if present. Sometimes, $A \leq C_* B$ will be abbreviated to $A \Cleq B$. For instance, if $K,K' \in \Mesh$, we have
\begin{equation}
\label{adjacent-elements}
 K \cap K' \neq \emptyset
 \quad\implies\quad
 \snorm{K} \Cleq \snorm{K'}
 \text{ and }
  h_K \Cleq \rho_{K'}.
\end{equation}

The linear space of (possibly) \emph{discontinuous piecewise polynomials} over $\Mesh$ with degree $\leq\degree$ is
\begin{equation*}
\label{Poly-ell}
 \Polyell{\degree}
 :=
 \{ s \in \Leb{\Omega}\mid
  \forall K \in \Mesh \; s_{|K} \in \Poly{\degree}(K) \},
\quad
 \degree \in \N_0. 
\end{equation*}
We shall need the following notation for discontinuities or jumps associated with functions from $\Polyell{\degree}$. Given an interior face $F \in \FacesMint$, let $K_1, K_2 \in \Mesh$ be the two elements such that $F=K_1 \cap K_2$. The ordering of $K_1$ and $K_2$ is arbitrary but fixed. For any function $v$ such that $v_{|K_j}$, $j=1,2$, have traces on $F$,
we define its jump across $F$ by
\begin{subequations}
\label{jumps}
\begin{equation}
\label{Jump}
\Jump{v}{F} (x) := 
v_{|K_1}(x) - v_{|K_2}(x),
\qquad
x \in F.
\end{equation}
The fact that the sign of $\Jump{v}{F}$ depends on the ordering of $K_1$ and $K_2$ will be insignificant to our discussion. Similarly, if  $n_{K_j}$ denotes the outward unit normal vector of $\partial K_j$, $j=1,2$, and $w$ is a suitable vector field, the jump of its normal component across $F$ is
\begin{equation}
\label{Jump-normal}
\Jump{w \cdot n}{F} (x) :=  
w_{|K_1} (x) \cdot n_{K_1} + w_{|K_2}(x) \cdot n_{K_2},
\qquad
x \in F,
\end{equation}
which is insensitive to the ordering of $K_1$ and $K_2$. It will be convenient to extend these definitions to boundary faces.  Given  $F \in \FacesMbnd$, let $K \in \Mesh$ be the element such that $F = K \cap \partial \Domain$ and set
\begin{equation}
 \Jump{v}{F}(x) := v_{|K}(x)
\quad\text{and}\quad
 \Jump{w \cdot n}{F}(x) := w_{|K}(x) \cdot n_K,
\qquad
 x \in F,
\end{equation}
where again we assume that the involved traces exist.
\end{subequations}

In this notation, the space of \emph{continuous piecewise polynomials} with degree $\leq\degree$ and vanishing trace reads
\begin{equation*}
\label{S-ell}
 \Lagr{\degree} 
 :=
 \SobH{\Domain} \cap \Polyell{\degree}
 =
 \{ s \in \Polyell{\degree} \mid
  \forall F \in \FacesM \; \Jump{s}{F} \equiv 0  \},
\quad
 \degree \in \N.
\end{equation*}
Denoting by  $\LagNodM{\degree} := \bigcup_{K \in \Mesh} \LagNod{\degree}(K)$ the Lagrange nodes of $\Mesh$, the point evaluations at the interior ones $\LagNodMint{\degree} := \{ z \in \LagNodM{\degree} \mid z \in \Omega \}$ form a set of degrees of freedom for $\Lagr{\degree}$, thanks to the interplay of the nestedness of the Lagrange nodes and the fact that $\Mesh$ is face-to-face. The associated nodal basis $\{\Phi^p_z\}_{z\in\LagNodMint{\degree}}$ is given by $\Phi^p_z(y) = \delta_{zy}$ for all $y\in\LagNodMint{\degree}$. The support of each $\Phi^p_z$ is the local star $\omega_z:= \bigcup_{K' \ni z} K'$, where we have
\begin{equation}
\label{Lagrange-basis:scaling}
 c_{d,\degree} |K'|^{\frac{1}{2}} h_{K'}^{-1}
 \leq
 \norm{\nabla \Phi^\degree_z}_{\Leb{K'}}
 \leq
 C_{d,\degree} |K'|^{\frac{1}{2}} \rho_{K'}^{-1}.
\end{equation}
thanks to the fact that Lagrange elements of order $p$ are affine equivalent.
Moreover, all supports of basis functions associated with an element $K\in\Mesh$ are contained in the patch $\omega_K := \bigcup_{K'\cap K \neq \emptyset} K'$.

The spaces $\Polyell{\degree}$ and $\Lagr{\degree}$ are connected by the following \emph{projection} $\AvgOper{\degree} : \Polyell{\degree} \to \Lagr{\degree}$ based upon evaluating at Lagrange nodes. For every interior node $z\in\LagNodMint{\degree}$, fix some element $K_z \in \Mesh$ containing $z$ and set
\begin{equation}
\label{Averaging-def}
 \textstyle
 \AvgOper{\degree} \sigma
 :=
 \sum_{z\in\LagNodMint{\degree}} \sigma_{|K_z} (z) \Phi^p_z,
\qquad
 \sigma \in \Polyell{\degree}.
\end{equation}
Clearly, $\AvgOper{\degree}\sigma(z) = \sigma (z)$ whenever $\sigma$ is continuous at $z\in\LagNodMint{\degree}$ and so $\AvgOper{\degree}$ is actually a projection onto $\Lagr{\degree}$. The operator $\AvgOper{\degree}$ can be seen, on the one hand, as a restriction of Scott-Zhang interpolation \cite{Scott.Zhang:90} defined for broken $H^1$-functions and, on the other hand, as a simplified variant of nodal averaging in that it requires only one evaluation per degree of freedom. Nodal averaging is used in various nonconforming contexts, see, e.g., Brenner \cite{Brenner:96}, Karakashian/Pascal \cite{Karakashian.Pascal:03}, Oswald \cite{Oswald:93}. It provides conformity along with the following bound, whose splitting is in the spirit of Brenner \cite{Brenner:03}. We provide a proof for the sake of completeness.

\begin{lemma}[An $H^1_0$-bound for simplified nodal averaging]
\label{L:averaging}
Let $p\in\N$, $\sigma \in \Polyell{\degree}$ piecewise polynomial, $K \in \Mesh$ some element, and $z \in \LagNod{\degree}(K)$ a Lagrange node. If $z \not\in \partial K$, then $\AvgOper{\degree}\sigma(z) = \sigma_{|K}(z)$, else
\begin{equation*}
%
 \snorm{\sigma_{|K}(z) - \AvgOper{\degree}\sigma(z)}
 \leq
	%
 \sum \limits_{F \ni z}
  \dfrac{1}{\snorm{F}} \snorm{\int_F \Jump{\sigma}{F}}
 + C_{d,p} \sum \limits_{K' \ni z} 
  \dfrac{h_{K'}}{\snorm{K'}^{\frac{1}{2}}} \norm{\nabla \sigma}_{\Leb{K'}},
\end{equation*}
where $F$ and $K'$ vary in $\FacesM$ and $\Mesh$, respectively.
\end{lemma}

\begin{proof}
If $z\not\in\partial K$, then the non-overlapping of elements in $\Mesh$ implies that $K_z$ in the definition of $\AvgOper{\degree}$ has to coincide with $K$ and the `then'-part of the claim is verified.  In order to show  the `else'-part, we start by claiming that, for any $z \in \partial K$,
\begin{equation}
\label{Averaging-est-1}
 \snorm{\sigma_{|K}(z) - \AvgOper{\degree}\sigma(z)}
%
\leq
\sum \limits_{F  \ni z} 
\snorm{\Jump{\sigma}{F}(z)}. 
\end{equation}
To verify this, we shall exploit that $\Mesh$ has face-connected stars in the sense of \cite{Veeser:16}, distinguishing the cases $z\in\Omega$ and $z\in\partial\Omega$.
If $z\in\Omega$ is an interior node, we choose a path $(K_j')_{j=0}^n$ in $\omega_z$ such that $K_0' = K$, $K_n' = K_z$ and $ K_{j-1}' \cap K_j' =: F_j \in \FacesMint$ for $ j=1, \dots n$. Then \eqref{Averaging-est-1} follows by bounding the telescopic sum $\sigma_{|K}(z) - \AvgOper{\degree}(z) = \sum_{j=1}^{n} \sigma_{|K_{j-1}}(z) - \sigma_{|K_j}(z)$ with the triangle inequality, independently of the choice of the path and $K_z$.
If $z\in\partial\Omega$ a boundary node, we proceed similarly but terminate the path with an element $K_b\in\Mesh$ that has a boundary face $F \in \FacesMbnd$ and use the identity $\sigma_{|K_b}(z) - A_p(z) = \sigma_{|K_b}(z) = \Jump{\sigma}{F}(z)$.

To derive the claimed inequality from \eqref{Averaging-est-1}, we need to bound each jump at $z$ suitably.  To this end, we consider again two cases, $F \in \FacesMint$ and $F \in \FacesMbnd$, and start with the first case. Let $K_1, K_2 \in \Mesh$ be the two elements such that $F = K_1 \cap K_2$. Inserting the face means $f_j := |F|^{-1} \int_F \sigma_{|K_j}$ as well as the element means
$k_j := |K_j|^{-1} \int_{K_j} \sigma$ and using an inverse estimate in $\Poly{\degree}(F)$, we deduce
\begin{equation}
\label{Averaging-inv-est}
 \snorm{\Jump{\sigma}{F}(z)}
 \leq
 \dfrac{1}{\snorm{F}} \snorm{\int_F \Jump{\sigma}{F}}
 + 
 \sum_{j=1,2} \left(
  | f_j - k_j |
  +
  \dfrac{C_{d,\degree}}{|F|^{\frac{1}{2}}} \norm{\sigma_{|K_j} -k_j}_{\Leb{F}}
 \right).
\end{equation}  
For $j=1,2$, the trace identity, see, e.g., \cite[Proposition~4.2]{Veeser.Verfuerth:09}, gives
\begin{equation*}
%
 \snorm{f_j-k_j} 
 \leq
 \dfrac{h_{K_j}}{d \snorm{K_j}} \norm{\nabla \sigma}_{L^1(K_j)}
 \leq
 \dfrac{h_{K_j}}{d \snorm{K_j}^{\frac{1}{2}}} \norm{\nabla \sigma}_{\Leb{K_j}}, , 
\end{equation*}
while \cite[Lemma~3]{Veeser:16}, which is a combination of the trace identity and the Poincar\'e inequality, provides
\begin{equation*}
%
  \snorm{F}^{-\frac{1}{2}}
  \norm{\sigma_{|K_j} - k_j}_{\Leb{F}} 
  \leq
  \sqrt{ \dfrac{1}{\pi^2} + \dfrac{2}{\pi d} } 
   \dfrac{h_{K_j}}{\snorm{K_j}^{\frac{1}{2}}}
   \norm{\nabla \sigma}_{\Leb{K_j}}.
\end{equation*}
Inserting the last two inequalities in \eqref{Averaging-inv-est}, we arrive at
\begin{subequations}
\label{Averaging-est-2}
\begin{equation}
 \snorm{\Jump{\sigma}{F}(z)}
 \leq
 \dfrac{1}{\snorm{F}} \snorm{\int_F \Jump{\sigma}{F}}
 + 
 C_{d,\degree} \sum_{j=1}^{2} \dfrac{h_{K_j}}{\snorm{K_j}^{\frac{1}{2}}}
  \norm{\nabla \sigma}_{\Leb{K_j}}
\end{equation}
in this case. If, instead, $F \in \FacesMbnd$, we denote by $K\in\Mesh$ the element with $F = K \cap \partial \Domain$ and, similarly, using the means  $f := |F|^{-1} \int_F \sigma_{|K}$ and $k := |K|^{-1} \int_{K} \sigma$, obtain
\begin{equation}
 \snorm{\Jump{\sigma}{F}(z)} 
 \leq
 \dfrac{1}{\snorm{F}} \snorm{\int_F \Jump{\sigma}{F}}
  + 
  C_{d,\degree} \dfrac{h_{K}}{\snorm{K}^{\frac{1}{2}}}
   \norm{\nabla \sigma}_{\Leb{K}}
\end{equation}
\end{subequations}
Inserting \eqref{Averaging-est-2} into \eqref{Averaging-est-1} then finishes the proof.
\end{proof}

\subsection{A quasi-optimal Crouzeix-Raviart method}
\label{S:CR-qopt}
%
In order to prove the results illustrated in the introduction \S\ref{S:intro}, we consider the approximation with Crouzeix-Raviart elements of the Poisson problem
\begin{equation}
\label{Poisson}
- \Delta u = f \text{ in }\Omega,
\qquad
u = 0 \text{ on }\partial\Omega,
\end{equation}
where $\Omega$ and $\Mesh$ are as in \S\ref{S:setting-FEM}, with $d\geq2$ and $\#\Mesh>1$. A function $w: \Domain \to \R$ is piecewise $H^1$ over $\Mesh$ and we write $w \in H^1(\Mesh)$ whenever $w_{|K} \in H^1(K)$ for all $K \in \Mesh$. The piecewise gradient $\GradM$ acts on $w$ as follows: $(\GradM w)_{|K} := \nabla(w_{|K})$ for all $K \in \Mesh$. Introducing the bilinear form 
$a_\Mesh:H^1(\Mesh)\times H^1(\Mesh) \to \R$ by
\begin{equation}
\label{aext-M}
 a_\Mesh(w_1, w_2)
 :=
 \int_{\Domain} \GradM w_1 \cdot \GradM w_2,
\end{equation}
we want to apply Corollary~\ref{C:ME-right-inverse} with the following setting:
\begin{equation}
\label{CR-setting}
\begin{gathered}
 V = \SobH{\Domain},
 \quad
 S = \CRS{} = \left\lbrace s \in \Polyell{1} \mid 
 \forall F \in \FacesM \; \int_F \Jump{s}{F}  = 0 \right\rbrace,
\\
 \aext = a_\Mesh{}_{|\Vext\times\Vext} \text{ with }  \Vext = \SobH{\Domain} + \CRS{},
\end{gathered}
\end{equation}
where $\aext_{|V\times V}$ provides a weak formulation of $-\Delta$.
Before embarking on the construction of the smoothing operator $E$, let us recall some relevant properties of $\CRS{}$; see, e.g., \cite{Brenner.Scott:08}. The characterization of $\CRS{}$ in terms of jumps is a consequence of the midpoint rule: whenever $s\in\CRS{}$ and $F\in\Faces{K}$, then $\int_F s_{|K} = s(m_F)$, where $m_F$ is the midpoint of $F$. Hence, for all $s \in \CRS{}$, the integral mean value $\int_F s$, $F \in \FacesM$, is well-defined and vanishes if $F \in \FacesMbnd$. The bilinear form $\aext$ is therefore a scalar product and induces the norm $\norm{\cdot} = \norm{\GradM \cdot}_{\Leb{\Domain}}$. Moreover, the functionals $s \mapsto \int_F s$, $F \in \FacesMint$, form a set of degrees of freedom for $\CRS{}$. We write 
$\Psi_F$, $F \in \FacesMint$, for the associated nodal basis satisfying $\int_{F'} \Psi_F = \delta_{F,F'}$ for all $F,F' \in \FacesMint$. The support of each basis function $\Psi_F$ is the union $\omega_F$ of the two elements sharing $F$. Finally, we have $\CRS{} \cap \SobH{\Domain} = \Lagr{1}$, which is a strict subspace of $\CRS{}$ as $\#\Mesh>1$.

The next lemma characterizes the right inverses of the Crouzeix-Raviart projection $\Ritz_{\crs}$, i.e.\ the $\aext$-orthogonal projection of $\Vext$ onto $\CRS{}$. 
\begin{lemma}[Right inverses of CR projection]
\label{L:CR-Rinv}
Let $E: \CRS{} \to \SobH{\Domain}$ be a linear operator. Then we have 
\begin{equation*}
 \Ritz_{\crs} E = \id_{\CRS{}}
	\quad
 \Longleftrightarrow
	\quad
 \forall \sigma \in \CRS{}, F \in \FacesMint
  \;
 \int_F E \sigma = \int_F \sigma.
\end{equation*}
\end{lemma} 
\begin{proof}
For any $v \in \SobH{\Domain}$ and $s \in \CRS{}$, the definition of $\Ritz_{\crs}$ and piecewise integration by parts yields
\begin{equation*}
 0 
 =
 a_\Mesh  (s, v - \Ritz_{\crs} v)
 =
 \sum \limits_{K \in \Mesh}
  \int_{ \partial K} \dfrac{\partial s}{\partial n_K} (v - \Ritz_{\crs} v)
 = 
 \sum \limits_{F \in \FacesMint}
  \Jump{\GradM s \cdot n}{F}\int_F (v - \Ritz_{\crs} v) 
\end{equation*} 
thanks to the fact that $\GradM s$ is piecewise constant and $\int_F v = 0 = \int_{F} \Ritz_{\crs} v$ for every $F \in \FacesMbnd$. Since the orthogonal projection $\Ritz_{\crs} v$ is unique and the averages over interior faces are degrees of freedom for $\CRS{}$, we obtain that
\begin{equation*}
 \forall F \in \FacesMint \qquad
 \int_F \Ritz_{\crs} v 
 =
 \int_F v
\end{equation*}
uniquely determines $\Ritz_{\crs} v$. This characterization readily implies the claimed equivalence.
\end{proof}

The normalized face bubbles
\begin{equation}
\label{CR-face-bubble}
 \bar{\Phi}_F
 :=
 \frac{(2d)!}{d! \, |F|} \Phi_F
\quad\text{with}\quad
 \Phi_F
 :=
 \prod \limits_{z \in \LagNod{1}(F)} \Phi^1_z
 =
 \frac{1}{d^d} \Phi^d_{m_F},
\qquad
 F \in \FacesMint,
\end{equation}
may be viewed as $\SobH{\Domain}$-counterparts of the nodal basis functions $\Psi_F$, $F\in\FacesMint$. Indeed, they satisfy $\Psi_F\in\SobH{\Domain}$ and $\int_{F'} \Phi_F = \delta_{F,F'}$ for all $F'\in\FacesMint$ due to \eqref{intregation-bary-coord}. We thus readily see that the linear operator $\BubbOper: \CRS{} \to \SobH{\Domain}$ given by 
\begin{equation}
\label{CR-E1}
 \BubbOper \sigma
 :=
 \sum_{F \in \FacesMint} \left( \int_{F} \sigma \right) \bar{\Phi}_F
\end{equation}
is well-defined and a right inverse of the Crouzeix-Raviart projection $\Ritz_{\crs}$. Unfortunately, the bubble smoothing operator $\BubbOper$ is not uniformly stable under refinement; see Remark~\ref{R:E1-norm} below. We therefore introduce the following variant that is stabilized with simplified nodal averaging.

\begin{proposition}[Stable right inverse of CR projection]
\label{P:stable-CR-Rinv}
The linear operator  $E_1:\CRS{}\to\SobH{\Domain}$ given by
\begin{equation*}
 E_1\sigma := \AvgOper{1}\sigma + \BubbOper(\sigma - \AvgOper{1}\sigma),
\end{equation*}
is invariant on $\Lagr{1}$, a right inverse of the Crouzeix-Raviart projection $\Ritz_{\crs}$, and $\SobH{\Domain}$-stable with stability constant $ \leq C_{d,\Shape_\Mesh}$.
\end{proposition}

\begin{proof}
The linear operator $E_1$ well-defined owing to $R(\AvgOper{1}) = \Lagr{1} \subseteq \CRS{}$ and  provides $\SobH{\Domain}$-smoothing, because $\Phi_z^1\in\SobH{\Domain}$ for $z \in \LagNodMint{1}$ and $\Phi_F\in\SobH{\Domain}$ for $F \in \FacesMint$. Owing to $\AvgOper{1 | \Lagr{1}} = \id_{\Lagr{1}}$, we have
$E_1{}_{|\Lagr{1}} = \id_{\Lagr{1}}$ on the conforming part $\Lagr{1} = \CRS{} \cap \SobH{\Domain}$  of the Crouzeix-Raviart space. Furthermore, $E_1$ is a right inverse of the Crouzeix-Raviart projection in view of Lemma~\ref{L:CR-Rinv}. Indeed, by rearranging terms and since $\BubbOper$ preserves face means, we find
\begin{equation}
\label{CR-E-Rinv}
\int_F E_1\sigma
=
\int_F \BubbOper \sigma
+
\underbrace{\int_F (\AvgOper{1}\sigma - \BubbOper \AvgOper{1}\sigma)}_{=0}
=
\int_F \sigma.
\end{equation}

It remains to bound $\opnorm{E_1}{\CRS{}}{\SobH{\Domain}}$. Given $\sigma \in \CRS{}$, we may write
\begin{equation*}
\label{CR-E-bound1}
\norm{\nabla E_1 \sigma}_{\Leb{\Domain}}
\leq
\norm{\GradM \sigma}_{\Leb{\Domain}}
+
\norm{\GradM(\sigma - \AvgOper{1} \sigma)}_{\Leb{\Domain}}
+
\norm{\nabla \BubbOper(\sigma - \AvgOper{1} \sigma)}_{\Leb{\Domain}}
\end{equation*}
so that  we have to bound the second and third term of the right-hand side by the first one. In both cases, we first establish a local bound for $K \in \Mesh$. For the second term, we combine Lemma \ref{L:averaging}, \eqref{adjacent-elements}, and \eqref{Lagrange-basis:scaling} to derive
\begin{equation}
\label{CR-E-bound2}
\begin{split}
\|\nabla (\sigma & - \AvgOper{1} \sigma)\|_{\Leb{K}} 
\leq
\sum \limits_{z \in \LagNod{1}(K)}
\snorm{\sigma_{|K}(z) - \AvgOper{1}\sigma(z)} 
\norm{\nabla\Phi_z^1}_{\Leb{K}}
\\
&\leq 
C_{d} \sum \limits_{z \in \LagNod{1}(K)} \;
\sum \limits_{K' \in \Mesh, K' \ni z}
\dfrac{h_{K'}}{\rho_K}\dfrac{\snorm{K}^{\frac{1}{2}}}{\snorm{K'}^{\frac{1}{2}}} \norm{\nabla \sigma}_{\Leb{K'}}
\Cleq
\norm{\GradM \sigma}_{\Leb{\omega_K}}.
\end{split}
\end{equation}
For the third term, inserting $\int_F \Phi^1_z = d^{-1}|F|$ and \eqref{CR-face-bubble} into \eqref{CR-E1} yields 
\begin{equation*}
\BubbOper(\sigma - \AvgOper{1} \sigma)_{|K} 
=
\frac{(2d)!}{d! \, d^{d+1}}
\sum \limits_{F \in \Faces{K}}
\sum \limits_{z \in \LagNod{1}(F)}
\big[ \sigma_{|K}(z) -  \AvgOper{1}\sigma(z) \big]  \Phi^d_{m_F}.
\end{equation*}
Hence, another combination of Lemma \ref{L:averaging}, \eqref{adjacent-elements}, and \eqref{Lagrange-basis:scaling} leads to
\begin{equation}
\label{CR-E-bound3}
\norm{\nabla \BubbOper (\sigma - \AvgOper{1} \sigma)}_{\Leb{K}}
\Cleq
\norm{\GradM \sigma}_{\Leb{\omega_K}}.
\end{equation}
We conclude by summing \eqref{CR-E-bound2} and \eqref{CR-E-bound3} over all mesh elements $K\in\Mesh$, observing that the number of elements in each star $\omega_K$ is $\leq C_{d,\Shape_\Mesh}$.
\end{proof}

Setting $E=E_1$ in \eqref{mod-CR-for-Poisson}, we obtain a \emph{new Crouzeix-Raviart method}, which we refer to as $\app_{\crs}$. Notice that the assembling of its load vector is computationally feasible in the following sense:
\begin{itemize}
\item it suffices to know the evaluations $\langle f, \Phi_z^1 \rangle$, $z \in \LagNodMint{1}$, and $\langle f, \Phi_F \rangle$, $F \in \FacesMint$,
\item it is local in that $\mathrm{supp}\, E_1 \Psi_F \subseteq \omega_{K_1} \cup \omega_{K_2}$, where $K_1,K_2 \in \Mesh$ are the two elements containing the interior face $F \in \FacesMint$. 
\end{itemize}
The method $\app_\crs$ distinguishes from the classical Crouzeix-Raviart method by the following property.

\begin{theorem}[Quasi-optimality of $\app_\crs$]
\label{T:MCR-qopt}
The method $\app_\crs$ is a $\norm{\GradM\cdot}$-quasi-optimal nonconforming Galerkin method for \eqref{Poisson} with $\Cqopt\leq C_{d,\Shape_\Mesh}$.
\end{theorem}

\begin{proof}
Notice that $\app_{\crs} = (\CRS{}, b, E_1)$, where $b$ is the restriction of $\aext$ in \eqref{CR-setting} to $\CRS{}\times\CRS{}$. Thus, the claim follows by using Proposition \ref{P:stable-CR-Rinv} in Corollary~\ref{C:ME-right-inverse}. 
\end{proof}

The following two remarks clarify that the single ingredients for $E_1$ are not suitable smoothing operators, thereby underlining their complementary roles.

\begin{remark}[Instability of bubble smoothing]
\label{R:E1-norm}
The right inverse $\BubbOper$ is not uniformly $\SobH{\Domain}$-stable under refinement. To see this, let $\Mesh$ be a mesh of $\Domain = (0,1)^2$ the elements of which have diameter $h>0$ and consider the function $\sigma := \sum_{F \in \FacesMint} \Psi_F$. Then $\sigma=1$ in all elements except those touching $\partial\Omega$, while $\BubbOper\sigma$ oscillates between $0$ and $1$ in all elements. Accordingly, $\bar{\Phi}_F = d^{-d}\Phi_{m_F}^d$, \eqref{Lagrange-basis:scaling}, and $h^{-1} \Cgeq \snorm{\nabla\Psi_F}$ give
\begin{equation*}
 \norm{\nabla \BubbOper\sigma}_{\Leb{\Domain}}
 \Cgeq
 \#\Mesh
 \Cgeq
 h^{-1} \#\{ K \in \Mesh \mid K \cap \partial \Domain \neq \emptyset  \}
 \Cgeq
 h^{-1} \norm{\GradM \sigma}_{\Leb{\Domain}}.
\end{equation*}
\end{remark}
\begin{remark}[Inconsistency of (simplified) nodal averaging]
\label{R:smooth-avg}
The use of smoothing operator $\AvgOper{1 | \CRS{}}$ in \eqref{mod-CR-for-Poisson} does not lead to full algebraic consistency and so in particular not to quasi-optimality. In fact, since $\dim \CRS{} > \dim \Lagr{1}$, the kernel $N(\AvgOper{1 | \CRS{}})$ is non-trivial. Moreover, as $\AvgOper{1 | \CRS{}}$ is not $\aext$-orthogonal, $N(\AvgOper{1 | \CRS{}})$ and $\Lagr{1}$ are not $\aext$-orthogonal. Consequently, we can find  $\sigma \in \CRS{}$ which is $\aext$-orthogonal to $\Lagr{1}$ and such that $s := \AvgOper{1} \sigma \neq 0$. Then $s \in \Lagr{1} = \CRS{} \cap \SobH{\Domain}$ and $b(s,\sigma) = 0 \neq a(s, \AvgOper{1}\sigma)$, which contradicts full algebraic consistency.
\end{remark}

\subsection{Quasi-optimal Crouzeix-Raviart-like methods of arbitrary order}
\label{S:GL-qopt}
%
In this section we generalize the quasi-optimal Crouzeix-Raviart method $\app_{\crs}$ of \S\ref{S:CR-qopt} to arbitrary fixed order $\degree\geq2$. To this end, let $\Omega$ and $\Mesh$ be as in \S\ref{S:setting-FEM}, $d\geq2$, $\#\Mesh>1$, and, this time, we want to apply Corollary~\ref{C:ME-right-inverse} with
\begin{equation}
\label{GLp-setting}
\begin{gathered}
 V = \SobH{\Domain},
\\
 \Lagr{\degree} \subseteq S \subseteq \GL{\degree} :=
 \left\lbrace  s \in \Polyell{\degree} \mid
  \forall F \in \FacesM, \, q \in \Poly{\degree-1}(F) \;
   \int_F \Jump{s}{F} q = 0
 \right\rbrace,
\\
  \aext = a_\Mesh{}_{|\Vext\times\Vext} \text{ with }  \Vext = V + S 
\end{gathered}
\end{equation}
and $a_\Mesh$ from \eqref{aext-M}. For any $d\geq2$, the space $\GL{1}$ coincides with the Crouzeix-Raviart space $\CRS{}$ from \S\ref{S:CR-qopt}. If $d=2$, then $\GL{\degree}$ is the Fortin-Soulie space \cite{Fortin.Soulie:83} for $\degree=2$, the Crouzeix-Falk space \cite{Crouzeix.Falk:89} for $\degree=3$, and, in general, the Gauss-Legendre space of Baran and Stoyan \cite{Baran.Stoyan:06} of order $\degree$. The last article provides a finite element basis of the Gauss-Legendre spaces, distinguishing odd and even polynomial degree $p$. For $d=3$, Fortin \cite{Fortin:85} for $p=2$ and Ciarlet et al.\ \cite{Ciarlet.Dunkl.Sauter:17} in general construct finite element bases for nonconforming subspaces of $\GL{\degree}$, strict in certain situations. In order to cover also these Crouzeix-Raviart-like spaces, we require in \eqref{GLp-setting} only $S\subseteq\GL{\degree}$.

Independently of the choice of $S$, we have that, for every $s \in S$, the moment $\int_F s q$ is well-defined for all $F \in \FacesM$ and all $q \in \Poly{\degree-1}(F)$ and vanishes whenever $F \in \FacesMbnd$. As a consequence, $\norm{\cdot} = \norm{\GradM\cdot}_{\Leb{\Domain}}$, which is induced by $\aext$, is a norm.

Let $\Ritz_S$ denote the $\aext $-orthogonal projection of $\Vext$ onto $S\subseteq\GL{\degree}$. Some right inverses thereof can be construct as follows.

\begin{lemma}[Right inverses of CR-like projections]
\label{L:GL-Rinv}	
Let $S\subseteq\GL{\degree}$ with $p\geq2$ and $E:S\to \SobH{\Domain}$ be a linear operator. If we have
\begin{equation}
\label{GL-Rinv}
%
%
 \int_F ( E \sigma) q
 =
 \int_F \sigma q,
\qquad
%
%
 \int_K  (E\sigma) r
 =
 \int_K \sigma r
\end{equation}
for all $\sigma \in S$, $F \in \FacesMint, q \in \Poly{\degree-1}(F)$ and $ K \in \Mesh, r \in \Poly{\degree-2}(K)$, then $\Ritz_S E = \id_S$.
\end{lemma} 
\begin{proof}
Given $s, \sigma \in S \subseteq \GL{\degree}$, we integrate piecewise by parts and obtain
\begin{equation*}
\begin{split}
 a_\Mesh  (s, \sigma &- E\sigma)
 =
 \sum \limits_{K \in \Mesh}
  \left( 
   \int_{\partial K} \dfrac{\partial s}{\partial n_K} (\sigma - E\sigma)
   -
  \int_K \triangle s (\sigma - E\sigma )
 \right)
\\
 &=
 \sum \limits_{F \in \FacesMint} 
  \int_F \Jump{\GradM s \cdot n}{F}(\sigma - E\sigma)
   -
  \sum \limits_{K \in \Mesh}
   \int_K \triangle s (\sigma - E\sigma )
  =
  0
\end{split}
\end{equation*}
thanks to the hypotheses on $E$. Hence, $0 = \Ritz_S(\sigma - E \sigma) = \sigma - \Ritz_S E \sigma$. 
\end{proof}

Let us construct such a smoothing operator by following the lines of the construction of $E_1$ in \S\ref{S:CR-qopt}. In order to define a higher order bubble smoother, we employ local weighted $L^2$-projections associated to faces and elements. For every interior face $F\in\FacesMint$, let $Q_F:\Leb{F} \to \Poly{\degree-1}(F)$ be given by
\begin{equation}
\label{Q_F}
 \forall q \in \Poly{\degree-1}(F)
\quad
 \int_{F} (Q_Fv) q \, \Phi_F = \int_{F} v q,
\end{equation}
where $\Phi_F \in \Lagr{d}$ is the face bubble function of \eqref{CR-face-bubble} with $\mathrm{supp}\,\Phi_F = \omega_F$, and, for every mesh element $K \in \Mesh$, let $Q_K: \Leb{K} \to \Poly{\degree-2}(K)$ be given by
\begin{equation}
\label{Q_K}
  \forall r \in \Poly{\degree-2}(K)
 \quad
 \int_{K} (Q_Kv) r \, \Phi_K = \int_{K} v r,
\end{equation}
where $\Phi_K := \prod_{z\in\LagNod{1}(K)} \Phi^1_z \in \Lagr{d+1}$ is the element bubble function with $\mathrm{supp}\,\Phi_K = K$. This leads to the global bubble operators
\begin{equation*}
 \BubbOper[\Mesh,\degree] v
 :=
 \sum_{K \in \Mesh} (Q_K v) \Phi_K,
\quad
 \BubbOper[\FacesM,\degree] v
 :=
 \sum_{F \in \FacesMint} \sum_{z \in \LagNod{\degree-1}(F)}
  (Q_F v)(z) \Phi_z^{\degree-1} \Phi_F,
\end{equation*}
where $\BubbOper[\FacesM,\degree]$ incorporates an extension by means of Lagrange basis functions. Their combination provides a right inverse of $\Ritz_S$.

\begin{lemma}[Higher order bubble smoother]
\label{L:GL-bubble-smoother}
For any $p\geq2$, the linear operator $\BubbOper[\degree]:\GL{\degree} \to \SobH{\Domain}$ defined by
\begin{equation*}
 \BubbOper[\degree] \sigma
 :=
 \BubbOper[\FacesM,\degree] \sigma + \BubbOper[\Mesh,\degree](\sigma - \BubbOper[\FacesM,\degree] \sigma)
\end{equation*}
satisfies \eqref{GL-Rinv} and the local stability estimate
\begin{equation*}
 \norm{\nabla \BubbOper[\degree] \sigma}_{\Leb{K}}
 \leq
 \frac{C_{d,\degree}}{\rho_K} \left(
 \sup_{r\in\Poly{\degree-2}(K)}
  \frac{\int_K \sigma r}{\norm{r}_{\Leb{K}}}
 + \sum_{F\in\Faces{K}}
  \frac{|K|^{\frac{1}{2}}}{|F|^{\frac{1}{2}}}
  \sup_{q\in\Poly{\degree-1}(F)} \frac{\int_F \sigma q}{\norm{q}_{\Leb{F}}}
 \right).
\end{equation*}
\end{lemma}

\begin{proof}
The operator $\BubbOper[\degree]$ is well-defined, because in particular the right-hand sides of \eqref{Q_F} are well-defined moments of any $\sigma \in \GL{\degree}$. Moreover, it maps into $\SobH{\Domain}$, since $\Phi_F \in \SobH{\Domain}$ for $F \in \FacesMint$ and  $\Phi_K \in \SobH{\Domain}$ for $K \in \Mesh$.

In order to verify \eqref{GL-Rinv}, let $\sigma \in S$ and consider, first, an interior face $F \in \FacesMint$ and $q \in \Poly{\degree-1}(F)$. In view of $\Phi_{K'}{}_{|F} = 0$ for $K' \in\Mesh$ and $\Phi_{F'}{}_{|F} = 0$ for $F'\neq F$, \eqref{Q_F} gives
\begin{equation*}
 \int_F (\BubbOper[\degree] \sigma) q
 =
 \int_F (Q_F \sigma) \Phi_F q
 =
 \int_F \sigma q.
\end{equation*}
Second, let $K \in \Mesh$ and $r \in \Poly{\degree-2}(K)$. Here, thanks to $\Phi_{K'}{}_{|K} = 0$ for $K'\neq K$, \eqref{Q_K} leads to
\begin{equation*}
 \int_K (\BubbOper[\degree] \sigma) r
 =
 \int_K (\BubbOper[\FacesM,\degree] \sigma ) r
  + \int_K Q_K (\sigma - \BubbOper[\FacesM,\degree] \sigma) \Phi_K r
 =
 \int_K \sigma r.
\end{equation*}

Finally, let us verify the stability estimate. Employing inverse estimates in $\Poly{\degree+d-1}(K)$ and $\Poly{\degree-1}(F)$ as well as $0 \leq \Phi_K \leq 1$ and \eqref{intregation-bary-coord}, we derive
\begin{equation}
\label{Ep:stability-in-Q}
\begin{aligned}
 \norm{\nabla \BubbOper[\degree] \sigma}_{\Leb{K}}
 &\leq
 C_{d,\degree} \rho_K^{-1} \norm{\BubbOper[\degree] \sigma}_{\Leb{K}}
\\
 &\leq
 C_{d,\degree} \frac{|K|^{\frac{1}{2}}}{\rho_K |F|^{\frac{1}{2}}} \norm{Q_F\sigma}_{\Leb{F}}
 +
  \frac{C_{d,\degree}}{\rho_K} \norm{Q_K\sigma}_{\Leb{K}}.
\end{aligned}
\end{equation}
Moreover, another inverse estimate in every $\Poly{\degree-2}(K)$ yields
\begin{equation*}
 \norm{Q_K\sigma}_{\Leb{K}}^2
 \leq
 C_{d,\degree} \int_K |Q_K\sigma|^2 \Phi_K
 =
 C_{d,\degree} \int_K \sigma Q_K\sigma,
\end{equation*}
whence
\begin{equation}
\label{Q_K-stability}
 \norm{Q_K\sigma}_{\Leb{K}}
 \leq
 C_{d,\degree}
 \sup_{r\in\Poly{\degree-2}(K)} \frac{\int_K \sigma r}{\norm{r}_{\Leb{K}}}.
\end{equation}
A similar argument in every $\Poly{\degree-1}(F)$ gives
\begin{equation}
\label{Q_F-stability}
\norm{Q_F\sigma}_{\Leb{F}}
\leq
 C_{d,\degree}
 \sup_{q\in\Poly{\degree-1}(F)} \frac{\int_F \sigma q}{\norm{q}_{\Leb{F}}}.
\end{equation}
We then obtain the stability estimate by inserting \eqref{Q_K-stability} and \eqref{Q_F-stability} into \eqref{Ep:stability-in-Q}.
\end{proof}

Stabilizing the bubble smoother $\BubbOper[\degree]$ with simplified nodal averaging $\AvgOper{\degree}$, we obtain a smoothing operator with the desired properties.

\begin{proposition}[Stable right inverses of CR-like projections]
\label{P:stable-GL-Rinv}
Let $\degree\geq2$ and $\Lagr{\degree} \subseteq S \subseteq \GL{\degree}$. The linear operator $E_{\degree}:S\to\SobH{\Domain}$ given by
\begin{equation*}
 E_{\degree}\sigma
 :=
 \AvgOper{\degree}\sigma + \BubbOper[\degree](\sigma - \AvgOper{\degree}\sigma)
\end{equation*}
is invariant on $\Lagr{\degree}$, a right inverse of the Crouzeix-Raviart-like projection $\Ritz_S$, and $\SobH{\Domain}$-stable with stability constant $\leq C_{d,\degree,\Shape_\Mesh}$.
\end{proposition}

\begin{proof}
We follow the lines of the proof of Proposition \ref{P:stable-CR-Rinv}  and easily check that $E_p$ is well-defined, provides $\SobH{\Domain}$-smoothing and is invariant on $\Lagr{\degree}$. Arguing as in \eqref{CR-E-Rinv} for any $F\in\FacesMint$ and any $q \in \Poly{\degree-1}(F)$ as well as for mesh element $K \in \Mesh$ and $r \in \Poly{\degree-2}(K)$, we find that that $E_{\degree}$ is a right inverse of $\Ritz_S$ onto $S$. 

It remains to bound $\opnorm{E_{\degree}}{S}{\SobH{\Domain}}$ appropriately. We let $\sigma \in S$ and write
\begin{equation*}
\label{GL-E-bound1}
\norm{\nabla E_{\degree} \sigma}_{\Leb{\Domain}}
\leq
\norm{\GradM \sigma}_{\Leb{\Domain}}
+
\norm{\GradM(\sigma - \AvgOper{\degree} \sigma)}_{\Leb{\Domain}}
+
\norm{\nabla \BubbOper[\degree](\sigma - \AvgOper{\degree} \sigma)}_{\Leb{\Domain}}.
\end{equation*}
To bound the second and third term, fix a mesh element $K \in \Mesh$. For the second term, we argue as in \eqref{CR-E-bound2}, with the polynomial degree $1$ replaced by $p$, and obtain
\begin{equation}
\label{GL-E-bound2}
 \norm{\nabla(\sigma - \AvgOper{\degree} \sigma)}_{\Leb{K}}
 \Cleq
 \norm{\GradM\sigma}_{\Leb{\omega_K}}.
\end{equation}
Regarding the third term, \eqref{intregation-bary-coord} gives
\begin{equation*}
  \sup_{r\in\Poly{\degree-2}(K)}
  \frac{\int_K (\sigma - \AvgOper{\degree}\sigma)r}{\norm{r}_{\Leb{K}}}
  \leq
  C_{d,\degree} \snorm{K}^{\frac{1}{2}}
   \sum_{z \in \LagNod{\degree}(\partial K)}
  	\snorm{\sigma_{|K}(z) - \AvgOper{\degree}\sigma(z)}
\end{equation*}
and, for every $F \in \Faces{K}$,
\begin{equation*}
 \sup_{q\in\Poly{\degree-1}(F)}
 \frac{\int_F (\sigma - \AvgOper{\degree}\sigma)q}{\norm{q}_{\Leb{F}}}
 \leq
 C_{d,\degree} \snorm{F}^{\frac{1}{2}}
  \sum_{z \in \LagNod{\degree}(F)}
   \snorm{\sigma_{|K}(z) - \AvgOper{\degree}\sigma(z)}.
\end{equation*}
Employing the stability estimate of Lemma \ref{L:GL-bubble-smoother}, the last two estimates and then Lemma \ref{L:averaging}, we derive 
\begin{equation}
\label{GL-E-bound3}
 \norm{\nabla \BubbOper (\sigma - \AvgOper{\degree} \sigma)}_{\Leb{K}}
\Cleq
\norm{\GradM \sigma}_{\Leb{\omega_K}}.
\end{equation}
Then summing \eqref{GL-E-bound2} and \eqref{GL-E-bound3} over all mesh elements $K \in \Mesh$ finishes the proof, as for Proposition \ref{P:stable-CR-Rinv}.
\end{proof}

We let $M_S$ denote the \emph{new Crouzeix-Raviart-like method of arbitrary fixed order} combining the setting \eqref{GLp-setting} with the smoothing operator $E_{\degree}$ in Proposition \ref{P:stable-GL-Rinv}. We have $M_S = (S,\aext_{|S \times S}, E_{\degree})$ and its discrete problem reads:
\begin{equation}
\label{GL-discrete-pb}
 U_S \in S
 \quad\text{such that}\quad
 \forall \sigma \in S \;
 \int_{\Omega} \GradM U_S \cdot \GradM\sigma
 =
 \langle f, E_{\degree} \sigma \rangle.
\end{equation}
Concerning the computational feasibility of $E_{\degree}$, notice that 
\begin{itemize}
\item it suffices to know the evaluations $\langle f, \Phi^{\degree}_z\rangle$ for $z \in \LagNodMint{\degree}$ as well as $\langle f, \Phi^{\degree-1}_z\Phi_F\rangle$ for $F\in\FacesMint$, $z \in \LagNod{\degree-1}(F)$, and $\langle f, \Phi^{\degree-2}_z\Phi_K\rangle$ for $K\in\Mesh$, $z \in \LagNod{\degree-2}(K)$,
\item $E_{\degree}$ is local in that, if $\omega$ is the support of a basis function $\Phi$ from \cite{Ciarlet.Dunkl.Sauter:17,Fortin:85,Baran.Stoyan:06}, then $\omega$ is a mesh element, a pair or a star of elements and $ \text{supp}\, E\Phi \subset \cup_{K\subset\omega} \omega_K$,
\item the operators $Q_F$ and $Q_K$ in \eqref{Q_F} and \eqref{Q_K} can be implemented by means of matrices which are precalculated on a reference element and, for $d=2$ and $Q_F$, can be diagonalized with the help of Legendre polynomials.
\end{itemize}
In contrast to the methods in
\cite{Ciarlet.Dunkl.Sauter:17,Fortin:85,Baran.Stoyan:06},
method $M_S$ enjoys the following property.
\begin{theorem}[Quasi-optimality of $M_S$]
\label{T:GL-qopt}
For any $\degree\geq2$ and any subspace $S$ with $\Lagr{\degree}\subseteq S \subseteq \CRS{\degree}$, the method $M_S$ is a $\norm{\GradM\cdot}$-quasi-optimal nonconforming Galerkin method for the Poisson problem \eqref{Poisson} with quasi-optimality constant $\leq C_{d,\degree,\Shape_\Mesh}$.
\end{theorem}

\begin{proof}
Use Proposition~\ref{P:stable-GL-Rinv} in Corollary \ref{C:ME-right-inverse}.
\end{proof}

\subsection{A quasi-optimal Morley method}
\label{S:MR-qopt}
%
This section constructs a quasi-optimal Morley method for the `biharmonic equation' with clamped boundary conditions,
\begin{equation}
\label{biharmonic_pb}
 \Delta^2 u = f \text{ in }\Omega,
\quad
 u = 0 \text{ and } \partial_n u = 0 \text{ on }\partial\Omega,
\end{equation}
where $\Omega$ and $\Mesh$ are as in \S\ref{S:setting-FEM}, $d=2$, and $\#\Mesh>1$. Defining $H^2(\Mesh)$ and the piecewise Hessian $\HessM$ similar to $H^1(\Mesh)$ and $\GradM$, we set
\begin{equation*}
%
 a_\Mesh(w_1, w_2) 
 :=
 \int_{\Domain} \HessM w_1 : \HessM w_2,
 \qquad
 w_1, w_2 \in H^2(\Mesh),
\end{equation*}
and aim at applying Corollary \ref{C:ME-right-inverse} with the following setting:
\begin{equation}
\label{MR-setting}
\begin{gathered}
 V=\SobbH,
\\
 S = \MRS{} 
 := 
 \left\{ s \in \Polyell{2} \mid
  s \text{ is cont.\ in } \LagNodM{1}, \,
  s_{|\LagNodMbnd{1}} = 0, \,
  \forall F \in \FacesM \; \int_F \Jump{\partial_n s}{F} = 0 
 \right\},
\\
 \aext = a_{\Mesh}{}_{|\Vext \times \Vext}
 \quad\text{with}\quad
 \Vext := \SobbH + \MRS{}, \qquad \qquad
\end{gathered}
\end{equation} 
where $\aext_{|V\times V}$ provides a weak formulation of $\Delta^2$ and $\MRS{}$ is the Morley space \cite{Morley:68} over $\Mesh$. In order to recall some useful properties of $\MRS{}$, let $n_F$ and $t_F$ be normal and tangent unit vectors  for every edge $F \in \FacesM$, with arbitrary but fixed orientation. The functionals $s \mapsto s(z)$, $z \in \LagNodMint{1}$, and $s \mapsto \int_F\nabla s \cdot n_F$, $F \in \FacesMint$, are well-defined for any $s\in\MRS{}$ and determine it. Furthermore, $\int_F \nabla s \cdot t_F$ and so also $\int_F \nabla s = |F| \nabla s (m_F)$ are well-defined and vanish if $F \in \FacesMbnd$. Hence, $\aext$ induces the norm $\norm{\HessM \cdot}_{\Leb{\Domain}}$ on $\Vext$.

\begin{remark}[Poor conforming part]
\label{R:conformingMR}
The conforming part $\MRS{}\cap\SobbH$ of the Morley space can be quite small, thereby providing only poor approximation properties; cf.\ de Boor and DeVore  \cite[Theorem 3]{deBoor.DeVore:83}. We illustrate this with an extreme example. Given any $n\in \N$, subdivide $\Domain = (0,1)^2$ into $n^2$ squares of equal size and obtain $\Mesh$ by inserting in each square the diagonal parallel to the line $\{ (x,x) \mid x \in \R \}$. Then $\MRS{}\cap\SobbH = \{ 0 \}$.
\end{remark}

We refer to the $\aext$-orthogonal projection of $\Vext$ onto $\MRS{}$ as the Morley projection $\Ritz_{\mrs}$. As before, the first step is to describe right inverses thereof.

\begin{lemma}[Right inverses of Morley projection]
\label{L:MR-Rinv}
Let $E:\MRS{} \to \SobbH$ be a linear operator. Then $\Ritz_{\mrs} E = \id_{\MRS{}}$ if and only if, for all $\sigma \in \MRS{}$,	
\begin{equation}
\label{MR-Rinv}
 \forall z \in \LagNodMint{1} \;
	E\sigma(z) = \sigma(z)
\quad\text{and}\quad
 \forall F \in \FacesMint \;
  \int_F \nabla E\sigma \cdot n_F
  =
 \int_F \nabla \sigma \cdot n_F.
\end{equation}
\end{lemma} 
\begin{proof}
Let us first characterize $\Ritz_{\mrs} v$ for any $v \in \SobbH$.  Defining $\sigma \in \MRS{}$ by
\begin{equation}
\label{MR-Ritz}
 \forall z \in \LagNodMint{1} \;
 \sigma(z) = v(z)
\quad\text{and}\quad
 \forall F \in \FacesMint \;
 \int_F \nabla \sigma \cdot n_F
 =
 \int_F \nabla v \cdot  n_F,
\end{equation}
we have	$\int_F \nabla v = \int_F\nabla \sigma$. Thus, integrating piecewise by parts, we infer
\begin{equation*}
 \forall s \in \MRS{} 
 \quad
 \aext(s, \sigma - v)
 =
 \sum \limits_{K \in \Mesh}
  \sum \limits_{F \in \Faces{K}}
	\int_F \Hess(s_{|K}) n_K \cdot \nabla (\sigma -v)
 =
 0
\end{equation*}
because $\HessM s$ is piecewise constant on $\Mesh$. Since the Morley projection of $v$ is unique, we derive that $\sigma = \Ritz_{\mrs} v$ and \eqref{MR-Ritz} characterizes $\Ritz_{\mrs} v$. This characterization readily yields the claimed equivalence.
\end{proof}

In order to construct such a right inverse that is stable under refinement, we again mimic the approach of \S\ref{S:CR-qopt}. Technical difficulties arise from the stronger regularity requirement $E \sigma \in \SobbH$; in particular, neither $\AvgOper{2}$ nor $\BubbOper[2]$ are applicable.
%
In order to replace the former, we employ the Hsieh-Clough-Tocher (HCT) element \cite{Clough.Tocher:65}. Given any $K \in \Mesh$, let $\Mesh_K$ be the triangulation obtained by connecting each vertex of $K$ with its barycenter $m_K$ and set 
\begin{equation*}
%
 \HCT 
 :=
 \{ s \in C^1(\overline{\Domain})  \mid 
  \forall K \in \Mesh \;
   s_{|K} \in C^1(K) \cap \Poly{3}(\Mesh_K),
   s = \partial_n s = 0 \text{ on }\partial\Omega
 \}.
\end{equation*}
Then $\HCT \subseteq \SobbH$ and every element $s \in \HCT$ is uniquely determined  by the values $s(z)$, $\nabla s(z)$ at the Lagrange nodes $z \in \LagNodMint{1}$ and $\nabla s (m_F) \cdot n_F $ at the midpoints $m_F$ of the interior edges $F \in \FacesMint$; see \cite{Brenner.Scott:08}. We denote the associated nodal basis by $\Upsilon_z^j$ with $z \in \LagNodMint{1}$, $j \in \{ 0, 1, 2\}$ and $\Upsilon_F$ with $F \in \FacesMint$, where $\Upsilon_z^0$ corresponds to $s(z)$, $\Upsilon_z^j$ to $\partial_j s(z)$ for $j=1,2$, and $\Upsilon_F$ to $\nabla s(m_F)\cdot n_F$. For every $z\in \LagNodMint{1}$, choose a fixed $K_z\in \Mesh$ containing $z$ and define $\AvgOper{\hct} : \MRS{ } \to \HCT{}$ by
\begin{equation}
\label{HCT-smoother}
 A_{\hct}\sigma
 :=
 \sum_{z \in \LagNodMint{1}} \left(
  \sigma(z) \Upsilon_z^0
   + \sum_{j=1}^2 \partial_{j} \big( \sigma_{|K_z} \big)(z) \Upsilon_z^j
 \right)
 + \sum_{F \in \FacesMint}
  \frac{\partial\sigma}{\partial {n_F}}(m_F) \Upsilon_{F}.
\end{equation}
In view of the properties of the Morley space $\MRS{}$, simplified averaging is only applied to the partial derivatives at the vertices.

In order to ensure the stability of $\AvgOper{\hct}$, we derive counterparts for Lemma \ref{L:averaging} and \eqref{Lagrange-basis:scaling}. Regarding the latter, observe that \eqref{Lagrange-basis:scaling} is derived by means of affine equivalence, while HCT elements are not affine equivalent. 

\begin{lemma}[Scalings of averaged HCT basis functions]
\label{L:HCT-scaling}
There are constants $C_1,C_2$ such that, for any element $K \in \Mesh$, vertex $z\in\LagNod{1}(K)$ and $j\in\{1,2\}$, we have
\begin{equation*}
  \left| \int_F \nabla\Upsilon_z^j\cdot n_F \right|
 \leq
 C_1 \Shape_K \snorm{F}
\quad\text{and}\quad
 \norm{\Hess\Upsilon_z^j}_{\Leb{K}}
 \leq
 C_2 \Shape_K \frac{\snorm{K}^{\frac{1}{2}}}{\rho_K}.
\end{equation*}
\end{lemma}

\begin{proof}
In the vein of Ciarlet \cite[Theorem~46.2]{Ciarlet:91}, we employ a closely related finite element that is given by the 12 functionals $P(z)$, $\nabla P(z)\cdot(y-z)$ for $y,z\in\LagNod{1}(K)$ with $y\neq z$ and $\nabla P(m_F)\cdot(m_K-m_F)$ for $F \in \Faces{K}$ on $C^1(K) \cap \Poly{3}(\Mesh_K)$. We denote the corresponding nodal basis on $K$ by $\widetilde{\Upsilon}_z$, $\widetilde{\Upsilon}_z^y$, $z,y \in \LagNod{1}(K)$ with $y\neq z$, and $\widetilde{\Upsilon}_F$, $F\in\FacesMint$. Since this element is affine equivalent, a comparison with a reference element yields, for every of its nodal basis function $\widetilde{\Upsilon}$ on $K$,
\begin{equation}
\label{HCT':scaling}
 \left| \int_F \nabla\widetilde{\Upsilon}\cdot n_F \right|
 \leq
 C_1 \rho_K^{-1} \snorm{F}
\quad\text{and}\quad
 \norm{\Hess\widetilde{\Upsilon}}_{\Leb{K}}
 \leq
 C \rho_K^{-2} \snorm{K}^{\frac{1}{2}}.
\end{equation}
Fix $z\in\LagNod{1}(K)$ and consider $\Upsilon_z^j$, where $j\in\{1,2\}$. Exploiting its duality,
\begin{equation*}
 \Upsilon_z^j(y) = 0,
\quad
 \partial_k \Upsilon_z^j(y) = \delta_{yz}\delta_{jk},
\quad
 \nabla\Upsilon_z^j(m_F) \cdot n_F = 0	
\end{equation*}
for all $y \in \LagNod{1}(K)$, $k\in \{1,2\}$, and $F\in\Faces{K}$, in the nodal basis representation of the affine equivalent element yields
\begin{equation*}
 \Upsilon_z^j
 =
 \sum_{y \in \LagNod{1}(K)\setminus\{z\}} (y-z)\cdot e_j \widetilde{\Upsilon}_z^y
 - \frac{1}{4}
 \sum_{F\in\Faces{K}:F\ni z}
  t_F \cdot e_j \; (m_K-m_F) \cdot t_F \widetilde{\Upsilon}_F. 
\end{equation*}
Combining this identity with \eqref{HCT':scaling} completes the proof.
\end{proof}

Lemma~\ref{L:averaging} has the following counterpart, where $\snorm{\cdot}$ stands also for the Euclidean norm on $\R^2$ and the jumps of vector fields across edges are defined componentwise.
\begin{lemma}[An $H^2_0$-bound for simplified nodal averaging into HCT]
\label{L:HCT-averaging}
There is a constant $C$, such that, for any Morley function $\sigma \in \MRS{}$, element $K \in \Mesh$ and vertex $z \in \LagNod{1}(K)$, we have
\begin{equation*}
\label{HCT-averaging-est}
 \snorm{\nabla \sigma_{|K} (z) - \nabla \AvgOper{\hct} \sigma (z)}
 \leq
 C \sum \limits_{K' \in \Mesh, K' \ni z}
  \dfrac{h_{K'}}{\snorm{K'}^{\frac{1}{2}}}
  \norm{D^2 \sigma}_{\Leb{K'}}.
\end{equation*}
\end{lemma}
\begin{proof}
Replacing $\sigma$ with $\nabla \sigma$, follow the lines of the proof of Lemma \ref{L:averaging} and use that $\int_F \Jump{\nabla \sigma}{F} = 0$ for all $F \in \FacesM$.
\end{proof}

The operator $\AvgOper{\hct}$ incidentally fulfills the first part of \eqref{MR-Rinv}. Aiming at a right inverse of the form $\AvgOper{\hct} + \BubbOper[\partial_n](\id_{\MRS{}} - \AvgOper{\hct})$, we thus only need to adjust the means of the normal derivative across interior faces by a suitable $\SobbH$-bubble smoother $\BubbOper[\partial_n]$. To this end, we replace the face bubbles in the bubble smoother $\BubbOper[1]$ of \S\ref{S:CR-qopt} by the following ones inspired by Verf\"urth \cite{Verfuerth:96}.  Given any interior edge $F \in \FacesMint$, let $K_1, K_2 \in \Mesh$ be the two elements such that $F = K_1 \cap K_2$ and consider their barycentric coordinates $(\lambda_z^{K_i})_{z \in \LagNod{1}(K_i)}$, $i=1,2$, as first-order polynomials on $\R^2$. Then
\begin{equation*}
 \bar{\phi}_F
 :=
 \frac{30}{\snorm{F}} \phi_F
\quad\text{with}\quad
 \phi_F
  :=
 \begin{cases}
  \prod \limits_{z \in \LagNod{1}(F)}
  \left( \lambda_z^{K_1} \lambda_z^{K_2}\right)^2
   &\text{in } K_1 \cup K_2,
  \\
  0 &\text{in } \Domain \setminus(K_1 \cup K_2)
 \end{cases}
\end{equation*}
is an $\SobbH$-counterpart of the normalized face bubble $\bar{\Phi}_F$ from \eqref{CR-face-bubble} and
\begin{equation}
\label{MR-face-bubble}
 \bar{\Phi}_{n_F}
 :=
 \zeta_F \bar{\phi}_F
\quad\text{with}\quad
 \zeta_F(x) := (x-m_F) \cdot n_F, \; x \in \R^2,
\end{equation}
is in $\SobbH$ and satisfies
$
 \int_{F'} \nabla \bar{\Phi}_{n_F} \cdot n_{F'}
 =
 \int_{F'} n_F \cdot n_{F'} \bar{\phi}_{F}
 =
 \delta_{F,F'}
$ 
for all $F' \in \FacesMint$ thanks to \eqref{intregation-bary-coord}. Hence, the operator $\BubbOper[\partial_n] : \MRS{}+\HCT \to \SobbH$ given by
\begin{equation*}
\label{MR-bubble-smoother}
 \BubbOper[\partial_n] \sigma
 :=
 \sum_{F \in \FacesMint} 
  \left( \int_F \nabla \sigma \cdot n_F \right) \bar{\Phi}_{n_F}
\end{equation*} 
provides $\SobbH$-smoothing with
\begin{equation}
\label{MR-bubble-smoother:partial_n}
 \forall F \in \FacesMint \quad
 \int_{F} \nabla( \BubbOper[\partial_n]\sigma) \cdot n_F
 =
 \int_{F} \nabla\sigma \cdot n_F
\end{equation}
and the following scaling of its `basis functions'.

\begin{lemma}[Scaling of MR face bubble]
\label{L:scaling-Phi_n_F}
If $K,K'\in \Mesh$ are the two elements containing the interior edge $F\in\FacesMint$, there is a constant $C$ such that
\begin{equation*}
 \norm{\Hess\bar{\Phi}_{n_F}}_{\Leb{K}}
 \leq
 C \Shape_K^3 \Shape_{K'}^2 \frac{|K|^{\frac{1}{2}}}{\rho_K |F|}.
\end{equation*}
\end{lemma}

\begin{proof}
If we use an inverse inequality in $\Poly{9}(K)$ and $|\zeta_F| \leq h_K$ on $K$, we obtain
\begin{equation*}
  \norm{\Hess\bar{\Phi}_{n_F}}_{\Leb{K}}
  \leq
  C \rho_{K}^{-2} \norm{\bar{\Phi}_{n_F}}_{\Leb{K}}
  \leq
   C \rho_{K}^{-2} h_K |F|^{-1} \norm{\phi_F}_{\Leb{K}}.
\end{equation*}
Moreover, for any $z \in \LagNod{1}(F)$, we have
$ 
 |\lambda_z^{K'}|
 \leq
 h_{K} |\nabla \lambda_z^{K'}|
 \leq
 \Shape_{K} \snorm{F} \rho_{K'}^{-1}
 \leq
 \Shape_{K} \Shape_{K'}
$ in $K$.
Using this and \eqref{intregation-bary-coord}, we finish the proof with
\begin{equation*}
 \norm{\phi_F}_{\Leb{K}}
 \leq
 \Shape_{K}^2 \Shape_{K'}^2 \;
  \norm{\prod_{z\in\LagNod{1}(F)} (\lambda_z^K)^2}_{\Leb{K}}
 =
 \Shape_{K}^2 \Shape_{K'}^2 \frac{\snorm{K}^{\frac{1}{2}}}{\sqrt{180}}.
 \qedhere
\end{equation*}
\end{proof}

Owing to these auxiliary results, we obtain a stable right inverse as before.
\begin{proposition}[Stable right inverse of Morley projection]
\label{P:stable-MR-Rinv}
The linear operator  $E_{\mrs}:\MRS{}\to\SobbH$ given by
\begin{equation*}
 E_{\mrs}\sigma
 :=
 \AvgOper{\hct}\sigma
  + \BubbOper[\partial_n](\sigma - \AvgOper{\hct}\sigma)
\end{equation*}
is invariant on $\MRS{}\cap\SobbH$, a right inverse of the Morley projection $\Ritz_{\mrs}$, and $\SobbH$-stable with stability constant $\leq C_{\Shape_\Mesh}$.
\end{proposition}

\begin{proof}
The operator $E_{\mrs}$ is invariant on $\MRS{}\cap\SobbH$, because $\AvgOper{\hct}$ is invariant on $\MRS{} \cap HCT = \MRS{} \cap \SobbH $. In order to check that $E_{\mrs}$ is a right inverse of $\Ritz_{\mrs}$, we verify \eqref{MR-Rinv} of Lemma~\ref{L:MR-Rinv} and let $\sigma \in \MRS{}$. First, given a Lagrange node $z\in\LagNodMint{1}$, we have $\AvgOper{\hct}\sigma(z) = \sigma(z)$ and so
$E_{\mrs}\sigma(z) = \sigma(z)$. Second, given an interior edge $F\in\FacesMint$, we derive
$\int_F \nabla E_{\mrs} \sigma \cdot n_F = \int_F \nabla \sigma \cdot n_F$
as in \eqref{CR-E-Rinv} by means of \eqref{MR-bubble-smoother:partial_n}. 

We may finish the proof by bounding $\opnorm{\id_{\MRS{}} - E_{\mrs}}{\MRS{}}{\SobbH}$ appropriately. Let $\sigma 
\in \MRS{}$, fix a mesh element $K\in\Mesh$, and write
\begin{equation*}
%
 \norm{\Hess (E_{\mrs} \sigma - \sigma)}_{\Leb{K}}
 \leq
 \norm{\Hess(\sigma - \AvgOper{\hct} \sigma)}_{\Leb{K}} 
 +
 \norm{\Hess \BubbOper[\partial_n](\sigma - \AvgOper{\hct}\sigma)}_{\Leb{K}}.
\end{equation*}
For the first term on the right-hand side, we combine Lemmas \ref{L:HCT-scaling} and \ref{L:HCT-averaging} as their counterparts in \eqref{CR-E-bound2} and obtain
\begin{equation*}
%
 \norm{\Hess(\sigma - \AvgOper{\hct} \sigma)}_{\Leb{K}}
 \Cleq 
 \norm{\HessM \sigma}_{\Leb{\omega_K}}.
\end{equation*}
For the second term, we observe
\begin{multline*}
 \BubbOper[\partial_n](\sigma - \AvgOper{\hct}\sigma)
 =
\\
 \sum_{F\in\Faces{K}\cap\FacesMint} \,
 \sum_{z \in \LagNod{1}(F)} \, \sum_{j=1}^2
  \big[ \partial_j(\sigma_{|K})(z) - (\partial_j\AvgOper{\hct}\sigma)(z) \big]
  \left( \int_F \nabla \Upsilon_z^j \cdot n_F \right) \bar{\Phi}_{n_F}
\end{multline*}
in $K$. Consequently, Lemmas~\ref{L:HCT-scaling}, \ref{L:HCT-averaging}, and \ref{L:scaling-Phi_n_F} give
\begin{equation*}
 \norm{\Hess \BubbOper[\partial_n](\sigma - \AvgOper{\hct}\sigma)}_{\Leb{K}}
 \Cleq
 \norm{\HessM \sigma}_{\Leb{\omega_K}}
\end{equation*}
and we can finish the proof as for Proposition \ref{P:stable-CR-Rinv}.
\end{proof}

Let $M_{\mrs}$ denote the \emph{new Morley method} for the biharmonic problem \eqref{biharmonic_pb} that corresponds to the setting \eqref{MR-setting} with the smoothing operator $E_{\mrs}$ in Proposition~\ref{P:stable-MR-Rinv}. Then $M_{\mrs}= (\MRS{},\aext_{|\MRS{} \times \MRS{}}, E_{\mrs})$ and its discrete problem is
\begin{equation}
\label{MR-discrete-pb}
U_{\mrs} \in \MRS{}
\quad\text{such that}\quad
\forall \sigma \in \MRS{} \;
\int_{\Omega} \HessM U_{\mrs} : \HessM\sigma
=
\langle f, E_{\mrs} \sigma \rangle.
\end{equation}
The smoother $E_{\mrs}$ is computationally feasible in that 
\begin{itemize}
\item it suffices to know the evaluations $\langle f, \Upsilon_z^j\rangle$ for $z \in \LagNodMint{1}$, $j\in\{0,1,2\}$, and $\langle f, \Upsilon_F\rangle$  for $F\in\FacesMint$, as well as $\langle f, \bar{\Phi}_{n_F}\rangle$ for $F\in\FacesMint$,
\item $E_{\degree}$ is local: if $\omega$ is the support of a Morley basis function, then $\omega$ is a pair or a star of elements and $ \text{supp}\, E\Phi \subset \cup_{K\subset\omega} \omega_K$.
\end{itemize}
The approximation properties of $M_{\mrs}$ are superior to the original Morley method.

\begin{theorem}[Quasi-optimality of $M_{\mrs}$]
\label{T:MR-qopt}
The method $M_{\mrs}$ is a $\norm{\HessM\cdot}$-quasi-optimal nonconforming Galerkin method for the biharmonic problem \eqref{biharmonic_pb} with  quasi-opti\-ma\-li\-ty constant $\leq C_{\Shape_\Mesh}$.
\end{theorem}

\begin{proof}
Use Proposition~\ref{P:stable-MR-Rinv} in Corollary \ref{C:ME-right-inverse}.
\end{proof}

\begin{remark}[Alternative simplified nodal averaging into rHCT]
\label{R:rHCT}
One obtains a variant of $M_{\mrs}$ by replacing in $E_{\mrs}$ the simplified nodal averaging $A_{\hct}$ from \eqref{HCT-smoother} by
\begin{equation*}
 A_{\rhct}\sigma
 :=
 \sum_{z \in \LagNodMint{1}} \left(
 \sigma(z) \Theta_z^0
 + \sum_{j=1}^2 \partial_{j} (\sigma_{|K_z})(z) \Theta_z^j
 \right),
\end{equation*}
where $\Theta_z^j$, $z \in \LagNodMint{1}$, $j\in\{0,1,2\}$, are the nodal basis functions of the reduced HCT space from Ciarlet \cite{Ciarlet:78-1}. As Lemma \ref{L:HCT-scaling} carries over to the new basis and the reduced HCT space contains $\MRS{}\cap\SobbH$, this modification of $M_{\mrs}$ is also a quasi-optimal nonconforming Galerkin method for \eqref{biharmonic_pb} with quasi-opti\-ma\-li\-ty constant $\leq C_{\Shape_\Mesh}$.
\end{remark}

\subsection*{Acknowledgments}
We wish to thank Christian Kreuzer for having brought reference \cite{Badia.Codina.Gudi.Guzman:14} to our attention as well as Francesca Tantardini and R\"{u}diger Verf\"{u}rth for their contribution to the proof of Proposition~\ref{P:stable-CR-Rinv}.


\end{document}